\documentclass[11pt]{amsart}
\usepackage{extsizes}
\usepackage{fullpage}
\usepackage{amsfonts,xcolor,url,verbatim}
\usepackage{amssymb}
\usepackage{xcolor}
\usepackage{dsfont}

\usepackage{todonotes}
\usepackage{hyperref}

\usepackage[backend=bibtex, sorting=none, bibstyle=alphabetic, citestyle=alphabetic, sorting=nyt, maxbibnames=99, giveninits=true, isbn=false, url=false, doi=false, eprint=false]{biblatex}  
\renewbibmacro{in:}{}
\bibliography{references.bib}
\usepackage{amssymb, calrsfs, graphics, graphicx, enumerate, enumitem, url, xcolor, hyperref, mathrsfs, listings, comment}
\usepackage{tabularx}
\usepackage[ruled, lined, linesnumbered, longend]{algorithm2e}
\usepackage[justification=centering]{caption}
\newtheorem{theorem}{Theorem}[section]
\newtheorem{lemma}[theorem]{Lemma}

\newtheorem{proposition}[theorem]{Proposition}

\newtheorem{corollary}[theorem]{Corollary}

\numberwithin{equation}{section}

\theoremstyle{remark}
\newtheorem{remark}{Remark}

\renewcommand{\Re}{{\mathfrak{Re}}}
\renewcommand{\Im}{{\mathfrak{Im}}}

\lstdefinestyle{CStyle}{
    basicstyle=\footnotesize,
    breakatwhitespace=false,         
    breaklines=true,                 
    captionpos=b,                    
    keepspaces=true,                 
    numbers=left,                    
    numbersep=5pt,                  
    showspaces=false,                
    showstringspaces=false,
    showtabs=false,                  
    tabsize=2,
    language=C
}

\def\p{\mathfrak{p}}

\def\P{\mathfrak{P}}

\newcommand{\Gal}{\mathrm{Gal}}
\newcommand{\GL}{\mathrm{GL}}
\newcommand{\norm}{\mathrm{N}}
\newcommand{\tr}{\mathrm{tr}}
\newcommand{\GO}{\mathrm{GO}}

\title[Counting zeros of Artin $L$-functions]{Counting zeros of Artin $L$-functions}

\author[C. Bellotti]{Chiara Bellotti}
\address[Chiara  Bellotti]{School of Science\\
The University of New South Wales, Canberra, Australia}
\email{c.bellotti@unsw.edu.au}

\author[P.-J. Wong]{Peng-Jie Wong}
\address[Peng-Jie Wong]{National Sun Yat-Sen University\\Department of Applied Mathematics\\
Kaohsiung City, Taiwan}
\email{pjwong@math.nsysu.edu.tw}

\keywords{Artin $L$-functions, Hecke $L$-functions, zero-density estimates, explicit estimates}
\subjclass[2020]{Primary 11R42; Secondary 11M06, 11M26, 11Y35}

\thanks{P.J.W. is currently supported by the NSTC grant 114-2628-M-110-006-MY4. C.B. was supported by Australian Research Council Discovery Project DP240100186 and by an Australian Mathematical Society Lift-off Fellowship.}

\begin{document}

\begin{abstract}
In this article, assuming Artin's (holomorphy) conjecture, we establish an explicit asymptotic formula for the number of non-trivial zeros, up to any given height $T\ge 1$, of Artin $L$-functions. As a consequence, our result yields an unconditional explicit zero-counting formula for  Hecke $L$-functions over any number field. In addition, our result improves the recent work of Amberger on Dedekind and Riemann zeta functions and the previous work of Bennett-Martin-O’Bryant-Rechnitzer on Dirichlet $L$-functions for sufficiently large $T$.
\end{abstract}

\maketitle

\section{Introduction}

Let $L/K$ be a Galois extension of number fields with Galois group $G=\Gal(L/K)$. For any irreducible representation $\rho: G\rightarrow \GL(V)$, we let $\chi=\tr \rho$ denote its character, and we recall the Artin $L$-function attached to $\chi$ is defined by
\begin{equation}\label{def:artinfunction}
    L(s,\chi, L/K) = \prod_{\p} \det(I -\rho|^{V^{I_{\mathfrak{P}}}}(\sigma_{\mathfrak{P}})\norm(\p)^{-s}   )^{-1},
\end{equation}
for $\Re(s)>1$, where the product runs over primes $\p$ of $K$, $V^{I_{\mathfrak{P}}}$ is the subspace of $V$ fixed by the inertia group $I_{\mathfrak{P}}$ of a prime $\mathfrak{P}$ of $L$ above $\p$, and $\sigma_{\mathfrak{P}}$ is the Frobenius element at $\P$. It follows from the works of Artin,  Brauer, and Hecke that $L(s,\chi, L/K)$ has an analytic continuation to a meromorphic function on $\Bbb{C}$. Artin's (holomorphy) conjecture asserts that if $\chi$ is non-trivial, then $L(s,\chi, L/K)$ further extends to an entire function. 

When $\chi(1)=1$, Artin's conjecture is proven as $L(s,\chi, L/K)$ corresponds to a Hecke $L$-function. More generally, Artin showed that his conjecture holds for all representations induced from 1-dimensional representations. Furthermore, if $\rho$ is 2-dimensional and admits a solvable image, Artin's conjecture is valid by the works of Langlands \cite{Lan80} and Tunnell \cite{Tun81}. Moreover, when $\rho$ is odd 2-dimensional over $\Bbb{Q}$, Artin's conjecture follows from Serre's modularity conjecture proven by Khare and Wintenberger \cite{KH09-1,KH09-2}. Furthermore, Ramakrishnan 
\cite{Ra02} proved Artin's conjecture for solvable representations $\rho$ of $\GO(4)$-type. In general, Artin's conjecture remains open and is an active area of research within the Langlands programme. (For a more detailed discussion, we shall refer the interested reader to \cite{Wong2018} and the references therein.) 

It is worth noting that Artin $L$-functions are vast generalisations of Dedekind zeta functions and Hecke $L$-functions. In particular, when $L=K=\Bbb{Q}$, one recovers the Dedekind zeta function
$\zeta_K(s)$ of $K$, and the case $L=K=\Bbb{Q}$ yields the Riemann zeta function. Meanwhile, when $L=\Bbb{Q}(\zeta_q)$, the $q$-th cyclotomic field, they are Dirichlet $L$-functions (modulo $q$). 
Similar to Dedekind zeta functions and Hecke $L$-functions, the (non-trivial) zeros of Artin $L$-functions encode deep arithmetic information of $L/K$ and help one to establish the strongest effective forms of the Chebotarev density theorem under Artin's conjecture (see \cite{murty1988modular, Murty1997, Murty2018}).

An elementary but prominent question is to count the zeros of $L(s,\chi, L/K)$ in the critical strip $0<\Re(s)<1$. For $T\geq 0$, we set 
\[
N(T,\chi) = \# \{ \rho \in  \Bbb{C}  \mid  L(\rho,\chi, L/K) =0,\  0  <\beta <  1, \ |\gamma| \leq T\},
\]
counted with multiplicity if any multiple zeros appear. Indeed, to make the aforementioned effective results explicit, it is natural to require a precise determination of the implied constants for the estimate of $N(T,\chi)$.

Throughout, for a number field $F$, we will denote its degree and absolute discriminant by $n_F$ and $d_F$, respectively. 
The main objective of this article is to prove the following theorem.

\begin{theorem}\label{main-thm}
Let $L/K$ be a Galois extension of number fields, and let $\chi$ be a non-trivial irreducible character of $\Gal(L/K)$. Assume Artin's conjecture for $L(s,\chi,L/K)$. Put
\[
        A(\chi)=d_K^{\chi(1)}\mathrm N(\mathfrak {f}(\chi))
        \quad\text{and}\quad
        m_\chi=n_K\chi(1).
\]
For every $T\ge 1$, the following estimates hold
\[
\left|
N(T,\chi)
-
\frac{T}{\pi}
\log\left(
A(\chi)
\left(\frac{T}{2\pi e}\right)^{m_\chi}
\right)
\right|
\le
0.1892\left(
\log A(\chi)+m_\chi\log T
\right)
+
9.484\,m_\chi,
\]
and
\[
\left|
N(T,\chi)
-
\frac{T}{\pi}
\log\left(
A(\chi)
\left(\frac{T}{2\pi e}\right)^{m_\chi}
\right)
\right|
\le
0.194\left(
\log A(\chi)+m_\chi\log T
\right)
+
7.907\,m_\chi.
\]
\end{theorem}
Table~\ref{tab:furtherArtin} lists further admissible pairs $(C_1,C_2)$,
where $C_1$ is the coefficient of
$\log A(\chi)+m_\chi\log T$ and $C_2$ is the coefficient of $m_\chi$
in the estimate of Theorem~\ref{main-thm}.

\begin{table}[htbp]
\centering
\renewcommand{\arraystretch}{1.2}
\begin{tabular}{|c|c|}
\hline
$C_1$ & $C_2$ \\
\hline
 0.200 & 6.751 \\
\hline
 0.240 & 3.976 \\
\hline
 0.280 & 2.941  \\
\hline
\end{tabular}
\caption{Further admissible pairs \((C_1,C_2)\) in the estimate for \(N(T,\chi)\).}
\label{tab:furtherArtin}
\end{table}

\begin{remark}
 (i)   It is well-known that Artin's conjecture for $L(s,\chi,L/K)$ implies that  $L(s,\chi,L/K)$ belongs to the Selberg class. In addition, we note that general explicit bounds for the number of zeros of $L$-functions in the Selberg class were established by Paloj{\"a}rvi \cite[Corollary 5.3 and Remark 5.4]{Pa2019}, which implies a version of Theorem~\ref{main-thm}, with weaker $(C_1,C_2)$, under Artin's conjecture. 
 
 \noindent (ii) Our argument is a refinement of Amberger's method
\cite{amberger2026estimatingnumberzerosdedekind} on counting zeros of Dedekind zeta functions, which is rooted in Turing's work \cite{Tur53}, and extends it to a broader class of $L$-functions (that are not ``self-dual''). The main new ingredient is the introduction of an additional parameter $a_4$, extending the three-parameter approach used in \cite[Lemma~3.1]{amberger2026estimatingnumberzerosdedekind}. This extra parameter is the main source of the improvements in our final constants, especially the leading constant $C_1$.
\end{remark}

When $\chi$ is the trivial character $\chi_0$, the associated Artin $L$-function is the Dedekind zeta function of $K$, defined as
\begin{equation}\label{def-Dedekind-zeta}
\zeta_K (s) = \prod_{\p} (1-\norm(\p)^{-s}   )^{-1}
\end{equation}
for $\Re(s)>1$. Our argument also leads to the following improvement of the recent work of Amberger \cite{amberger2026estimatingnumberzerosdedekind} (indeed, the second estimate in the theorem below is always sharper than \cite[Thm.\,2.1]{amberger2026estimatingnumberzerosdedekind}).

\begin{theorem}\label{Dedekind:thm} Let $K$ be a number field of degree $n_K$ and absolute discriminant $d_K$.
Let
$$
N_K(T) = \# \{ \rho \in  \Bbb{C}  \mid  \zeta_K(s) =0,\  0  <\beta <  1, \ |\gamma| \leq T\}.
$$
For every $T\ge 1$, the following estimates hold:
\[
\left|
N_K(T)
-
\frac{T}{\pi}
\log\left(
d_K
\left(\frac{T}{2\pi e}\right)^{n_K}
\right)
\right|
\le
0.1892\left(
\log d_K+n_K\log T
\right)
+
9.484\,n_K
+
2.007,
\]
and 
\[
\left|
N_K(T)
-
\frac{T}{\pi}
\log\left(
d_K
\left(\frac{T}{2\pi e}\right)^{n_K}
\right)
\right|
\le
0.194\left(
\log d_K+n_K\log T
\right)
+
7.907\,n_K
+
2.001.
\]
\end{theorem}

\begin{remark}
     In \cite{KN12}, extending the arguments of Backlund \cite{Ba18},  McCurley \cite{McC84}, and  Rosser \cite{Ro41} to Dedekind zeta functions, Kadiri and Ng  showed that for $T\ge 1$, one has
\begin{equation}\label{counting-zeros}
 \Big| N_K (T)  -  \frac{T}{\pi} \log \Big( d_K\Big( \frac{T}{2\pi e}\Big) ^{n_K} \Big)  \Big|  \leq  
C_1 (\log d_K + n_K \log T) +C_2  n_K  +C_3,
\end{equation}
with admissible $(C_1,C_2,C_3)= (0.506,16.950, 7.663 )$; also, $C_1$ can be taken as small as $(\pi \log 2)^{-1}\approx 0.459$ at the expense of larger $C_2  n_K  + C_3$. This was improved by Trudgian \cite{Tr15}, who showed that the estimate \eqref{counting-zeros} is valid with $(C_1,C_2,C_3)= (0.316, 5.872, 3.655)$, and  the constant $C_1$ in \eqref{counting-zeros} could be made as small as $0.247$ with  $(C_2 ,C_3) =(8.851,  3.024)$. In \cite[Theorem 1.1]{HSW21-Dedekind}, adapting the techniques developed in \cite{BMOR20}, Hasanalizade-Shen-Wong showed that  for any $T\ge 1$, one has
\begin{equation}\label{main-bound-1}
\Big| N_K (T)  - \frac{T}{\pi} \log \Big( d_K \Big( \frac{T}{2\pi e}\Big)^{n_K}\Big)  + \frac{r_1}{4}\Big|
\le  0.22737  \log\Big(  \frac{d_K(T+2)^{n_K}}{(2\pi)^{n_K}}  \Big)  + 23.02528 n_K  +   4.51954,
\end{equation} 
where $r_1$ is the number of real places of $K$. Furthermore, Amberger \cite{amberger2026estimatingnumberzerosdedekind} improved all the previous results (including the ones in \cite[Table 1]{HSW21-Dedekind}) with admissible $(C_1,C_2,C_3)$ recorded in Table \ref{tab:comparison-amberger}.\footnote{The values of \(C_3\) reported in
\cite[Theorem~2.1 and Table~4]{amberger2026estimatingnumberzerosdedekind}
appear to rely on a sign convention for the \(a_3\)-term that is inconsistent
with the one used earlier in the proof. More precisely, the signs of the
\(a_3\)-terms in equations (3.11) and (3.12) do not appear to be compatible
with the kernel inequality in Lemma~3.1. Consequently, the values of \(C_3\)
listed in Table~4 may require further verification.}

\begin{table}[htbp]
\centering
\renewcommand{\arraystretch}{1.2}
\begin{tabular}{|c|c|c||c|c|c|}
\hline
\multicolumn{3}{|c||}{Amberger \cite{amberger2026estimatingnumberzerosdedekind}}
&
\multicolumn{3}{c|}{Our improvement} \\
\hline
$C_1$ & $C_2$ & $C_3$ &
$C_1$ & $C_2$ & $C_3$ \\
\hline
0.194 & 8.161 & 2.001 & 0.194 & 7.907 & 2.001 \\
\hline
0.200 & 6.803 & 2.001 & 0.200 & 6.751 & 2.015 \\
\hline
0.240 & 4.155 & 2.001 & 0.240 & 3.976 & 2.067 \\
\hline
0.280 & 3.055 & 2.001 & 0.280 & 2.941 & 2.127 \\
\hline
\end{tabular}
\caption{Comparison of explicit constants in the estimates for \(N_K(T)\).}
\label{tab:comparison-amberger}
\end{table}
\end{remark}

An important consequence of Theorem \ref{Dedekind:thm} is the following explicit bound for $N(T)$, the number of non-trivial zeros $\rho$, with $0<\Im(\rho) \le T$, of the Riemann zeta function $\zeta(s)$.

\begin{corollary}\label{improvedzeta}
Let $N(T)= \frac{1}{2} N_{\Bbb{Q}}(T)$. For all $T\ge 1$, the following estimates hold:
\[
\left|
N(T)
-
\frac{T}{2\pi}
\log\left(
\frac{T}{2\pi e}\right)
\right|
\le
0.0946\log T
+
5.746,
\]
and
\[
\left|
N(T)
-
\frac{T}{2\pi}
\log\left(
\frac{T}{2\pi e}\right)
\right|
\le 0.097\log T
+
4.954.
\]
In particular, the first estimate is sharper for $T\ge \exp(330)$. 
\end{corollary}

\begin{remark}
The study of $N(T)$ has a long and rich history. For the convenience of the reader, writing 
\begin{equation}\label{his-bound}
\left| N (T)  - \frac{T}{   2 \pi} \log \left( \frac{T}{2\pi e}\right)  \right|
 \le   C_1 \log T  + C_2   \log\log T  + C_3,
\end{equation}
for $T\ge T_0$, we summarise the advances that have been made in Table \ref{tableriemann} below.
\begin{table}[htbp] 
\centering
\begin{tabular}{ |c||c|c|c|c|   } 
 \hline  
  &  $C_1$ & $C_2$ & $C_3$ & $T_0$ \\ 
 \hline
 von Mangoldt \cite{vMo05} (1905) & 0.4320   & 1.9167  &  13.0788  &  28.5580     \\ 
 \hline
 Grossmann \cite{Gr13} (1913) & 0.2907   & 1.7862 &  7.0120 &  50  \\ 
 \hline 
 Backlund \cite{Ba18} (1918)  & 0.1370   & 0.4430  & 5.2250 & 200 \\ 
 \hline 
 Rosser \cite{Ro41} (1941)  & 0.1370   & 0.4430  & 2.4630 & 2\\ 
 \hline  
 Trudgian \cite{Trudgian20121053} (2012) &  0.1700& 0& 2.8730 & $e$\\
 \hline
 Trudgian \cite{Tr14-2} (2014)   & 0.1120   & 0.2780  & 3.3850 & $e$\\ 
 \hline 
 Platt--Trudgian \cite{PLATT2015842} (2015) & 0.1100 & 0.2900 & 3.165 & $e$\\
 \hline
Hasanalizade, Shen, and Wong \cite{HASANALIZADE2022219} (2022)   & 0.1038   & 0.2573  & 9.4925 & $e$\\ 
 \hline  
 Bellotti--Wong \cite{BeWo25} (2025) & 0.10076 & 0.24460 & 8.08344 & $e$\\
 \hline
  Amberger \cite{amberger2026estimatingnumberzerosdedekind} (2025) & 0.097 & 0 & 5.081 & 1 \\
 \hline
\end{tabular}
   \caption{Previous explicit bounds for $N(T)$ in \eqref{his-bound}}\label{tableriemann} 
\end{table}
\end{remark}

Now, let $\psi$ be a Hecke character (of finite order), and let $L(s,\psi)$ and $\mathfrak{f}(\psi)$ be its associated Hecke $L$-function and conductor, respectively. By class field theory, there is a Galois extension $L/K$ with a character $\chi\in\Gal(L/K)$ of degree 1 (i.e. $\chi(1)=1$) such that $L(s,\chi,L/K)= L(s,\psi)$ and $\mathfrak{f}(\chi)= \mathfrak{f}(\psi)$. (See \cite[Ch. VII, \textsection 10]{Neu99} for more details.) From this correspondence, applying Theorem \ref{main-thm}, we then derive the following zero-counting result for Hecke $L$-functions.

\begin{theorem}\label{th:dedekind}
Let $K$ be a number field of degree $n_K$ and absolute discriminant $d_K$. Let $\psi$ be a non-trivial Hecke character (of finite order), and let $\mathfrak{f}(\psi)$ be its conductor. Define
\[
N(T,\psi) = \# \{ \rho \in  \Bbb{C}  \mid  L(\rho,\psi) =0,\  0  <\beta <  1, \ |\gamma| \leq T\}.
\]
For every $T\ge 1$, the following estimates hold:
\[
\left|
N(T,\psi)
-
\frac{T}{\pi}
\log\left(
d_K\mathrm N(\mathfrak{f}(\psi))
\left(\frac{T}{2\pi e}\right)^{n_K}
\right)
\right|
\le
0.1892\left(
\log (d_K\mathrm N\mathfrak{f}(\psi))+ n_K\log T
\right)
+
9.484\, n_K
\]
and
\[
\left|
N(T,\psi)
-
\frac{T}{\pi}
\log\left(
d_K\mathrm N(\mathfrak{f}(\psi))
\left(\frac{T}{2\pi e}\right)^{n_K}
\right)
\right|
\le
0.194\left(
\log (d_K\mathrm N\mathfrak{f}(\psi))+ n_K\log T
\right)
+
7.907 n_K.
\]
 \end{theorem}

\begin{remark}
Previously, for any fixed totally imaginary field $K$, adapting Trudgian's argument in \cite{Tr15}, Grze{\'s}kowiak \cite[Theorem 1.1]{Gr17} proved a version of Theorem \ref{th:dedekind}, where the leading term in the bound is
$C_1\left(
 \frac{n_K}{2}\log (d_K\mathrm N\mathfrak{f}(\psi))+ n_K\log T
\right)$, with $C_1\ge (\pi \log 4)^{-1} \approx 0.229612$.
\end{remark}

When $K=\Bbb{Q}$, this theorem yields the following zero-counting result for Dirichlet $L$-functions.

\begin{corollary}\label{coro:dedekind}
Let $\chi$ be a non-trivial  Dirichlet character modulo $q$, and let $L(s,\chi)$ be the Dirichlet $L$-function of $\chi$. Set
$$
N_\chi(T) = \# \{ \rho \in  \Bbb{C}  \mid  L(\rho,\chi) =0,\  0  <\beta <  1, \ |\gamma| \leq T\}.
$$
For every $T\ge 1$, the following estimates hold:
\[
\left|
N_{\chi}(T)
-
\frac{T}{\pi}
\log\left(
\frac{qT}{2\pi e}
\right)
\right|
\le
0.1892
\log (q T)
+
9.484,
\]
and
\[
\left|
N_{\chi}(T)
-
\frac{T}{\pi}
\log\left(
\frac{qT}{2\pi e}
\right)
\right|
\le
0.194
\log (q T)
+
7.907. 
\]
\end{corollary}

\begin{remark}
(i) We remark that Theorem~\ref{th:dedekind} and Corollary~\ref{coro:dedekind} remain valid with
$(0.1892,9.484)$ and $(0.194,7.907)$ replaced by  any of the admissible pairs $(C_1,C_2)$ listed in
Table~\ref{tab:furtherArtin}.

\noindent (ii) The explicit bounds for $N_\chi(T)$ have been established by McCurley \cite{McC84}, Trudgian \cite{Tr15}, and then Bennett-Martin-O’Bryant-Rechnitzer \cite{BMOR20}. McCurley and Trudgian obtained bounds of the form
\[
\left|
N_{\chi}(T)
-
\frac{T}{\pi}
\log\left(
\frac{qT}{2\pi e}
\right)
\right|
\le
C_1
\log (q T)
+
C_2
\]
for positive constants $C_1$ and $C_2$. In \cite[Theorem 2.1]{McC84}, McCurley gave a general formulation for $C_1=C_1(\eta)$ and $C_2=C_2(\eta)$, as functions of a parameter $\eta\in (0,\frac{1}{2}]$, with all $C_1 > 1/\pi \log 2 > 0.45$. In \cite[Theorem 1]{Tr15}, Trudgian 
refined McCurley’s techniques and gave ten pairs of values $(C_1, C_2)$ with all $C_1\ge 0.247$; in his proof, it was asserted
that $C_1$ could be taken as small as $(\pi \log 4)^{-1} 
\approx 0.229612$. In \cite[Theorem 1.1]{BMOR20}, a slightly more complicated bound
\[
\left|
N_{\chi}(T)
- \left(
\frac{T}{\pi}
\log\left(
\frac{qT}{2\pi e}
\right) -\frac{\chi(-1)}{4}\right)
\right|
\le
0.22737 \log \frac{q(T+2)}{2\pi} + 2\log\left(1+\log \frac{q(T+2)}{2\pi}\right) - 0.5
\]
was established for $T \ge 5/7$ with  $\log \frac{q(T+2)}{2\pi}> 1.567$. As a direct consequence, in \cite[Corollary 1.2]{BMOR20}  Bennett-Martin-O’Bryant-Rechnitzer showed that $(C_1,C_2)$ can be taken as $(0.247, 6.894)$ and $(0.298 ,4.358)$, for $T \ge 5/7$, which improve all results of Trudgian as well as the parametric bound of McCurley. 
Moreover, it is worth remarking that from our main theorem, one can even take $(C_1,C_2)=(0.240, 3.976)$ for $T\ge 1$.
\end{remark}

\section{The main term and the gamma factor}\label{main}

\subsection{The main term}

For a number field $F$, we let  $n_F=[F:\Bbb{Q}]$ and $d_F$ denote its degree and absolute discriminant, respectively. Let $L/K$ be a Galois extension of number fields with Galois group $G$. In this section, we shall collect some preliminaries for Artin $L$-functions (mainly from \cite{martinet1977} and \cite[Ch. VII]{Neu99}). For any irreducible character $\chi$ of $G$, we let $L(s,\chi, L/K)$  be the Artin $L$-functions attached to $\chi$ defined in \eqref{def:artinfunction}. For $1\le j \le \dim V^{I_{\mathfrak{P}}}$, let $\lambda_{j,\p}$ be eigenvalues of $\rho|^{V^{I_{\mathfrak{P}}}}(\sigma_{\mathfrak{P}})$, and we set $\lambda_{j,\p}=0$ for $ \dim V^{I_{\mathfrak{P}}}< j\le  \dim V =\chi(1)$. Then we have
\begin{equation}\label{L-prod}
L(s,\chi, L/K)=\prod_{\p} \det(I -\rho|^{V^{I_{\mathfrak{P}}}}(\sigma_{\mathfrak{P}})\norm(\p)^{-s}   )^{-1}
= \prod_{\p}\prod_{1\le j\le \chi(1)} (1- \lambda_{j,\p} \norm(\p)^{-s}   )^{-1}
\end{equation}
for $\Re(s)>1$. Consequently, we have the following result.
\begin{lemma} \label{bd-Artin-L}
For $s=\sigma+it$ with $\sigma>1$, one has
\begin{align*}
\left(\frac{\zeta_K (2 \sigma)}{\zeta_K (\sigma)} \right)^{\chi(1)} \leq |L(s,\chi, L/K)|  
\le \zeta_K (\sigma)^{\chi(1)}
\leq \zeta (\sigma)^{n_K \chi(1)},
\end{align*}
where, as usual, $\zeta(s)$ denotes the Riemann zeta function. In particular, when $\chi$ is the trivial character of $G=\Gal(L/K)$, then
\begin{align*}
\frac{\zeta_K (2 \sigma)}{\zeta_K (\sigma)}  \leq |\zeta_K (s)| \leq \zeta (\sigma)^{n_K}
\end{align*}
for $s=\sigma+it$ with $\sigma>1$.
\end{lemma}
\begin{proof}
   Observe that
    $$
    \Big|\prod_{1\le j\le \chi(1)} (1- \lambda_{j,\p} \norm(\p)^{-s}   )^{-1}\Big|
   \le \prod_{1\le j\le \chi(1)} (1- \norm(\p)^{-\sigma}   )^{-1} \le  ((1- \norm(\p)^{-\sigma}   )^{-1})^{\chi(1)}
    $$
    and 
\begin{align*}
   \Big|\prod_{1\le j\le \chi(1)} (1- \lambda_{j,\p} \norm(\p)^{-s}   )^{-1}\Big|
   &\ge \prod_{1\le j\le \chi(1)} (1+ \norm(\p)^{-\sigma}   )^{-1} 
  \\& \ge  ((1+ \norm(\p)^{-\sigma}   )^{-1})^{\chi(1)}
   = \left(\frac{(1- \norm(\p)^{-2\sigma}   )^{-1}}{(1- \norm(\p)^{-\sigma}   )^{-1}} \right)^{\chi(1)} .
 \end{align*}
We then conclude the proof by recalling \eqref{L-prod} and the definition of $\zeta_K (s)$ in \eqref{def-Dedekind-zeta}.
\end{proof}
Furthermore, following \cite[pp. 10--11]{martinet1977}, for any prime $\mathfrak{P}$ of $L$ that is above $\p$, we define
$$
\chi(\sigma_{\mathfrak{P}}^m)
=\frac{1}{e} \sum_{g \mapsto \sigma_{\mathfrak{P}}^m} \chi(g),
$$
where $e$ is the ramification index of $\mathfrak{P}$, and the sum is over the $e$ elements $g\in D_\mathfrak{P}$ that map onto $\sigma_{\mathfrak{P}}^m \in D_\mathfrak{P}/I_\mathfrak{P}$. (Note that
the order of the inertia group $I_\mathfrak{P}$ equals $e$, so the sum ``takes an average".) One has 
$$
\log L(s,\chi,L/K)
=\sum_{\p}\sum_{m=1}^\infty \frac{\chi(\sigma_{\mathfrak{P}}^m)}{m\norm(\p)^{ms}},
$$
where, for $\Re (s)>1$, we take the branch defined by the absolutely convergent Dirichlet series. Consequently, by the fact that $|\chi(\sigma_{\mathfrak{P}}^m)
|\le \chi(1)$, we then have the following handy corollary.

\begin{corollary}\label{cor:artin-log-derivative-bounds}
Let $s=\sigma+i t$ with $\sigma>1$. Then
\[
|\log L(s,\chi,L/K)|
\le  \chi(1)\log \zeta_K (\sigma) \le
n_K\chi(1)\log\zeta(\sigma),
\]
where the branch of the logarithm is defined by the absolutely convergent
Euler product. Moreover, for every integer $k\ge 0$, one has
\[
\left|
\frac{d^k}{ds^k}
\frac{L'}{L}(s,\chi,L/K)
\right|
\le
n_K\chi(1)
(-1)^{k+1}
\frac{d^k}{d\sigma^k}
\left(\frac{\zeta'}{\zeta}\right)(\sigma),
\]
where the right-hand side is positive for $\sigma>1$. The same estimates
hold with $\chi$ replaced by $\bar\chi$.
\end{corollary}
\begin{remark}
The estimates in Lemma~\ref{bd-Artin-L} and
Corollary~\ref{cor:artin-log-derivative-bounds} are deliberately uniform in
the extension \(L/K\) and the character \(\chi\). In special situations, one
may be able to obtain sharper bounds by using more detailed information
about \(L/K\) or \(\chi\), or by exploiting the pairing between
\(L(s,\chi)\) and \(L(s,\bar\chi)\). This may be particularly relevant for
special classes of characters, such as quadratic or cubic characters. 
However, we do not pursue such refinements here, since the presented uniform estimates are simpler to state and apply, and are sufficient for the general bounds proved in this paper.
\end{remark}
Let $\mathfrak{f}(\chi)$ denote the Artin conductor of $\chi$, and set $A(\chi)= d_K^{\chi(1)} \norm(\mathfrak{f}(\chi))$. In addition, we set $\delta(\chi)=1$ if $\chi$ is the trivial character; otherwise $\delta(\chi)=0$. The completed function $\xi (s,\chi) =\xi(s,\chi, L/K)$ is defined by
\begin{equation}\label{def-xi-chi}
\xi (s,\chi)= (s(s-1))^{\delta(\chi)} A(\chi)^{s/2} \gamma_\chi (s) L (s,\chi, L/K), 
\end{equation}
where
\begin{equation}\label{def:gammachi}
    \gamma_\chi (s) = 
\Big( \pi^{-\frac{s}{2}} \Gamma \Big( \frac{s}{2}\Big)\Big)^{ c_0(\chi)}
\Big(\pi^{-\frac{s+1}{2}} \Gamma \Big( \frac{s+1}{2}\Big) \Big)^{c_1(\chi)} ,
\end{equation}
with $c_0(\chi),c_1(\chi) \in \Bbb{Z}_{\ge 0}$ such that
$$
c_0(\chi)+c_1(\chi)  = n_K \chi(1).
$$
By the works of Artin and Brauer (see, e.g., \cite[Ch. VII, \textsection 12]{Neu99}),  it is known that $\xi(s,\chi)$ satisfies the functional equation
\begin{equation}\label{FE-Artin}
 \xi (s,\chi) = W(\chi) \xi (1-s,\bar{\chi}) ,
\end{equation}
where $W(\chi)$ is the global root number of $L(s,\chi, L/K)$ such that $|W(\chi)|=1$. In addition, it is known that $A(\bar{\chi}) = A(\chi)$ and $\gamma_{\bar{\chi}}=\gamma_\chi$. Consequently, one has
\begin{equation}\label{FE-Artin-dual}
\xi (s,\chi)= \overline{\xi (\bar{s},\bar{\chi})}.
\end{equation}

\begin{remark}
Let $r_1$ and $r_2$ be the numbers of real and complex places, respectively, of $K$; note that $n_K = r_1 +2r_2$. 
For Dedekind zeta functions, namely, $\chi=\chi_0$ is the trivial character of $G=\Gal(L/K)$, we then have the completed zeta function $\xi_K(s) $ defined as
\begin{equation}\label{def-xi}
\xi_K(s)= s(s-1) d_K^{s/2} \gamma_K (s) \zeta_K (s),
\end{equation}
where 
$$
\gamma_K (s) = \Big(\pi^{-\frac{s+1}{2}} \Gamma \Big( \frac{s+1}{2}\Big) \Big)^{r_2} \Big( \pi^{-\frac{s}{2}} \Gamma \Big( \frac{s}{2}\Big)\Big)^{r_1 + r_2}.
$$
We recall that $\xi_K(s)$ extends to an entire function of order 1 and satisfies the functional equation
\begin{equation}\label{FE}
\xi_K(s) = \xi_K (1-s).
\end{equation}
\end{remark}
From now on, unless otherwise stated, we first assume Artin's 
conjecture for $L(s,\chi,L/K)$ attached to a non-trivial irreducible character $\chi$. Note that for such an instance, the functional equation \eqref{FE-Artin} implies Artin's conjecture for $L(s,\bar{\chi},L/K)$. Hence, $\xi(s,\chi)$ and $\xi(s,\bar{\chi})$ are entire functions of order one. The case of trivial character, for which $L(s,\chi,L/K)=\zeta_K(s)$, will be treated separately by using the completed Dedekind zeta-function $\xi_K(s)$.

Moreover, we also assume initially that $T$ is not the ordinate of a zero of either
$L(s,\chi,L/K)$ or $L(s,\bar{\chi},L/K)$. The general case follows by the usual limiting argument (see, e.g., \cite[p. 279,  Footnote  1]{HSW21-Dedekind}). Denote
\[ L_\chi(s):=L(s,\chi,L/K)\quad\text{and}\quad L_{\bar{\chi}}(s):=L(s,\bar{\chi},L/K).
\]
For the zero-counting argument, we will use the auxiliary symmetrized completed function
\[
\Xi_\chi(s):=\xi(s,\chi)\xi(s,\bar{\chi}).
\]
This is not a new completion of $L_\chi(s)$ but a useful function for
packaging the pair $\chi,\bar{\chi}$, which will ease the notation throughout the proof. Since
\[
A(\bar{\chi})=A(\chi),\qquad
\gamma_{\bar{\chi}}(s)=\gamma_\chi(s),\qquad
W(\chi)W(\bar{\chi})=1,
\]
we have
\begin{equation}\label{functeqXi}
     \Xi_\chi(s)=\Xi_\chi(1-s)=\overline{\Xi_\chi(\bar{s})}.
\end{equation}
Thus, $\Xi_\chi$ admits the same formal symmetry as completed Dedekind
zeta functions. Its zeros are the multiset union of the zeros of
$L_\chi(s)$ and $L_{\bar{\chi}}(s)$. Since complex conjugation gives a
multiplicity-preserving bijection between the zeros of $L_\chi(s)$ and
those of $L_{\bar{\chi}}(s)$, this multiset has twice as many zeros as
$L_\chi(s)$ alone. More precisely, let $\mathcal{Z}_\chi$ denote the multiset union of the
non-trivial zeros of $L_\chi(s)$ and $L_{\bar{\chi}}(s)$. If
$\rho=\beta+i\gamma\in\mathcal{Z}_\chi$, then
$1-\bar{\rho}=1-\beta+i\gamma$ also belongs to $\mathcal{Z}_\chi$, with the
same multiplicity.\\
Although $\Xi_\chi$ has twice as many zeros as $L_\chi$, the operators
defined below are normalized so as to count $N(T,\chi)$, not the full zero
count $N_{\Xi_\chi}(T)$. Equivalently, they are one half of the operators
one would obtain by applying the Dedekind-zeta argument directly to the function $\Xi_\chi$.\\
By the argument principle applied to $\xi(s,\chi)$, for any $d>\frac{1}{2}$, we have \[
N(T,\chi)=\frac{1}{2\pi i} \oint_{\mathcal{R}} \frac{\xi' (s,\chi)}{\xi (s,\chi)} ds,
\]
where $\mathcal{R}$ is the rectangular contour with vertices
$$
\frac{1}{2}-d-i T, \quad \frac{1}{2}+d-i T, \quad \frac{1}{2}+d+i T, \quad \frac{1}{2}-d+i T.
$$
Using the functional equation \eqref{FE-Artin}, we have
\begin{align*}
N(T,\chi)
&=\frac{1}{2\pi i} \int_{\mathcal{P}} \frac{\xi' (s,\chi)}{\xi (s,\chi)}+  \frac{\xi' (s,\bar{\chi})}{\xi (s,\bar{\chi})} ds
=\frac{1}{2\pi i} \int_{\mathcal{P}} \frac{\Xi'_\chi(s)}{\Xi_\chi(s)} ds,
\end{align*}
where $\mathcal{P}$ is the part of the contour $\mathcal{R}$ on $\Re(s)\ge 0$ (i.e., the path connecting $\frac{1}{2} - iT$, $\frac{1}{2} +d - iT$, $\frac{1}{2} +d + iT$, and $\frac{1}{2} +d$). In addition, by \eqref{functeqXi}, the symmetry of complex conjugation, we obtain
\begin{align*}
N(T,\chi)
&=\frac{1}{\pi i} \int_{\frac{1}{2} +d}^{\frac{1}{2} +d+iT} \frac{\Xi'_\chi(s)}{\Xi_\chi(s)} 
+ \frac{1}{\pi i}\int_{\frac{1}{2} +d+iT}^{\frac{1}{2} +iT} \frac{\Xi'_\chi(s)}{\Xi_\chi(s)} ds  .
\end{align*}
Hence, similar to \cite[Eq. (3.3)]{amberger2026estimatingnumberzerosdedekind} (which is rooted in a method of Turing  \cite[Sec. 4, Lemma 1]{Tur53}), we arrive at
\begin{align}\label{eq:Nchi-contour-reduction}
N(T,\chi)
&=
-\frac{1}{\pi}
\int_{1/2}^{1/2+d}
\left(
\Im\frac{\Xi_\chi'}{\Xi_\chi}(\sigma+iT)
-
\Im\frac{\Xi_\chi'}{\Xi_\chi}(\sigma+d+iT)
\right)
\,d\sigma
+
E_1(\Xi_\chi;T,d),
\end{align}
where
\begin{equation}\label{eq:E1-def}
E_1(\Xi_\chi;T,d)
=
\frac{1}{\pi}
\left(
2\Im\log \Xi_\chi\left(\frac12+d+iT\right)
-
\Im\log \Xi_\chi\left(\frac12+2d+iT\right)
\right).
\end{equation}
Equivalently,
\[
E_1(\Xi_\chi;T,d)
=
\mathcal{E}_1(\xi(s,\chi);T,d)
+
\mathcal{E}_1(\xi(s,\bar{\chi});T,d),
\]
where
\[
\mathcal{E}_1(F;T,d)
=
\frac{1}{\pi}
\left(
2\Im\log F\left(\frac12+d+iT\right)
-
\Im\log F\left(\frac12+2d+iT\right)
\right).
\]
By the Hadamard factorization for the entire function $\Xi_\chi$, which is of order one, we can write
\[
        \frac{\Xi_\chi'}{\Xi_\chi}(s)
        =
        B_\chi
        +
        \sum_{\rho\in\mathcal Z_\chi}
        \left(
        \frac{1}{\rho}
        +
        \frac{1}{s-\rho}
        \right),
\]
where $B_\chi\in\mathbb R$. This follows from \eqref{functeqXi}. Hence, for $s=\sigma+it$,
\[
\Im\frac{\Xi_\chi'}{\Xi_\chi}(s)
=
\sum_{\rho=\beta+i\gamma\in\mathcal{Z}_\chi}
\left(
\frac{\gamma-t}{(\sigma-\beta)^2+(t-\gamma)^2}
-
\frac{\gamma}{\beta^2+\gamma^2}
\right).
\]
Substituting this into \eqref{eq:Nchi-contour-reduction}, the terms involving
$\gamma/(\beta^2+\gamma^2)$ cancel, and we get
\begin{align}\label{eq:Nchi-zero-sum-before-pairing}
\begin{split}
&N(T,\chi)\\
&=
\frac{1}{\pi}
\int_{1/2}^{1/2+d}
\sum_{\rho=\beta+i\gamma\in\mathcal{Z}_\chi}
\left(
\frac{T-\gamma}{(\sigma-\beta)^2+(T-\gamma)^2}
-
\frac{T-\gamma}{(\sigma+d-\beta)^2+(T-\gamma)^2}
\right)
d\sigma
+
E_1(\Xi_\chi;T,d).
 \end{split}
\end{align}
For $b,t\in\mathbb{R}$ and $d>1/2$, define
\begin{equation}\label{eq:f-btd}
f(b,t;d)
=
2\tan^{-1}\left(\frac{b+d}{t}\right)
+
2\tan^{-1}\left(\frac{-b+d}{t}\right)
-
\tan^{-1}\left(\frac{b+2d}{t}\right)
-
\tan^{-1}\left(\frac{-b+2d}{t}\right).
\end{equation}
Pairing each zero $\rho=\beta+i\gamma$ with $1-\bar{\rho}$, and using the
symmetry of $\mathcal{Z}_\chi$, we obtain
\begin{equation}\label{eq:Nchi-f-sum}
N(T,\chi)
=
\frac{1}{2\pi}
\sum_{\rho=\beta+i\gamma\in\mathcal{Z}_\chi}
f\left(\frac12-\beta,T-\gamma;d\right)
+
E_1(\Xi_\chi;T,d).
\end{equation}
We now introduce a fourth parameter $a_4$. The operators $E_1,E_2,E_3$
are the same as in the three-parameter argument of
\cite{amberger2026estimatingnumberzerosdedekind}; the new contribution is
the operator $E_4$ defined below.\\
For a meromorphic function $F$, write
\[
D_F(s):=\frac{F'}{F}(s),\qquad
D_F'(s):=\left(\frac{F'}{F}\right)'(s),\qquad
D_F''(s):=\left(\frac{F'}{F}\right)''(s).
\]
Let
\[
        s_0:=\frac12+d+iT.
\]
Define
\begin{equation}\label{eq:E2-def-a4}
E_2(F;T,d)
:=
\frac{d a_1}{4}\Re D_F(s_0)
-
\frac{d^2a_2}{4}\Re D_F'(s_0),
\end{equation}
\begin{equation}\label{eq:E3-def-a4}
E_3(F;T,d)
:=
\frac{a_3}{4}\Im D_F'(s_0),
\end{equation}
and
\begin{equation}\label{eq:E4-def-a4}
E_4(F;T,d)
:=
-\frac{d^3a_4}{8}\Im D_F''(s_0).
\end{equation}
Here, $E_4$ is the only new operator. It corresponds to the term appearing in Lemma \ref{lem:cubic-kernel}.
\begin{lemma}\label{lem:cubic-kernel}
Let
\[
        I_3(x,t):=\frac{3x^2t-t^3}{(x^2+t^2)^3}.
\]
Then
\[
        I_3(x,t)=-\Im\frac{1}{(x+it)^3}.
\]
Moreover, for $F=\Xi_\chi$,
\[
        E_4(\Xi_\chi;T,d)
        =
        -\frac{d^3a_4}{8}
        \Im\left[
        \left(\frac{\xi'}{\xi}\right)''(s_0,\chi)
        +
        \left(\frac{\xi'}{\xi}\right)''(s_0,\bar\chi)
        \right].
\]
\end{lemma}

\begin{proof}
The identity
\[
        I_3(x,t)=-\Im (x+it)^{-3}
\]
follows by expanding
\[
        (x-it)^3=(x^3-3xt^2)-i(3x^2t-t^3).
\]
Since $\Xi_\chi(s)=\xi(s,\chi)\xi(s,\bar\chi)$, 
we have
\[
        D_{\Xi_\chi}''(s)
        =
        \left(\frac{\xi'}{\xi}\right)''(s,\chi)
        +
        \left(\frac{\xi'}{\xi}\right)''(s,\bar\chi).
\]
Substituting this identity into the definition
\[
        E_4(F;T,d)=-\frac{d^3a_4}{8}\Im D_F''(s_0)
\]
gives the claimed formula for $F=\Xi_\chi$.
\end{proof}
For $x,t\in\mathbb R$, put
\[
        R_1(x,t):=\frac{x}{x^2+t^2},
        \qquad
        R_2(x,t):=\frac{x^2-t^2}{(x^2+t^2)^2},
        \qquad
        I_2(x,t):=\frac{2xt}{(x^2+t^2)^2}.
\]
We say that the tuple $(d,a_1,a_2,a_3,a_4)$ is admissible if $d>1/2$ and, for every
$|b|\le 1/2$ and every $t\ne0$, one has
\begin{align}\label{eq:four-param-kernel-upper}
f(b,t;d)
&\le
\frac{\pi}{4}
\Big[
        d a_1\{R_1(d+b,t)+R_1(d-b,t)\} 
        +d^2a_2\{R_2(d+b,t)+R_2(d-b,t)\} \notag\\
&\quad
        +a_3\{I_2(d+b,t)+I_2(d-b,t)\} 
        +d^3a_4\{I_3(d+b,t)+I_3(d-b,t)\}
\Big].
\end{align}
The upper inequality implies the lower bound
\begin{align}\label{eq:four-param-kernel-lower}
f(b,t;d)
&\ge
\frac{\pi}{4}
\Big[
        -d a_1\{R_1(d+b,t)+R_1(d-b,t)\}
        -d^2a_2\{R_2(d+b,t)+R_2(d-b,t)\} \notag\\
&\quad
        +a_3\{I_2(d+b,t)+I_2(d-b,t)\}
        +d^3a_4\{I_3(d+b,t)+I_3(d-b,t)\}
\Big].
\end{align}
Indeed, this follows from \eqref{eq:four-param-kernel-upper} upon replacing
$t$ by $-t$, since $f(b,t;d)$, $I_2(x,t)$, and $I_3(x,t)$ are odd in $t$,
whereas $R_1(x,t)$ and $R_2(x,t)$ are even in $t$.

\begin{proposition}[Four-parameter zero-sum bound]\label{prop:four-param-zero-sum}
Assume that $(d,a_1,a_2,a_3,a_4)$ is admissible. Then
\begin{equation}\label{eq:f-sum-upper-a4}
\frac{1}{2\pi}
\sum_{\rho\in\mathcal{Z}_\chi}
f\left(\frac12-\beta,T-\gamma;d\right)
\le
E_2(\Xi_\chi;T,d)+E_3(\Xi_\chi;T,d)+E_4(\Xi_\chi;T,d),
\end{equation}
and
\begin{equation}\label{eq:f-sum-lower-a4}
\frac{1}{2\pi}
\sum_{\rho\in\mathcal{Z}_\chi}
f\left(\frac12-\beta,T-\gamma;d\right)
\ge
-E_2(\Xi_\chi;T,d)+E_3(\Xi_\chi;T,d)+E_4(\Xi_\chi;T,d).
\end{equation}
\end{proposition}

\begin{proof}
We prove the upper bound; the lower bound follows from
\eqref{eq:four-param-kernel-lower}. Put
\[
        b=\frac12-\beta,
        \qquad
        t=T-\gamma.
\]
For a zero $\rho=\beta+i\gamma$, we have
\[
        s_0-\rho=d+b+it.
\]
The zero paired with $\rho$ by the functional equation is
$1-\bar\rho=1-\beta+i\gamma$, and
\[
        s_0-(1-\bar\rho)=d-b+it.
\]
The $R_1$, $R_2$, and $I_2$ terms give exactly the operators $E_2$ and $E_3$,
as in the three-parameter argument of
\cite{amberger2026estimatingnumberzerosdedekind}. For the new cubic term,
Lemma \ref{lem:cubic-kernel} gives
\[
I_3(d+b,t)+I_3(d-b,t)
=
-\Im\left(
        \frac{1}{(s_0-\rho)^3}
        +
        \frac{1}{(s_0-(1-\bar\rho))^3}
\right).
\]
Since $\rho\mapsto 1-\bar\rho$ permutes $\mathcal Z_\chi$, summing over
$\rho\in\mathcal Z_\chi$ gives
\[
\sum_{\rho\in\mathcal Z_\chi}
\{I_3(d+b,t)+I_3(d-b,t)\}
=
-\Im D_{\Xi_\chi}''(s_0).
\]
Multiplying the pointwise inequality by $1/(2\pi)$, the cubic term therefore
contributes
\[
        -\frac{d^3a_4}{8}\Im D_{\Xi_\chi}''(s_0)
        =
        E_4(\Xi_\chi;T,d).
\]
This proves the upper bound.
\end{proof}
Combining \eqref{eq:Nchi-f-sum} with Proposition
\ref{prop:four-param-zero-sum}, and recalling that our operators are
normalized for the half-count $N_{\Xi_\chi}(T)/2=N(T,\chi)$, we obtain
\begin{equation}\label{eq:Eu-El-bound-a4}
E_\ell(\Xi_\chi;T,d)
\le
N(T,\chi)
\le
E_u(\Xi_\chi;T,d),
\end{equation}
where
\begin{equation}\label{eq:Eu-def-a4}
E_u(F;T,d)
:=
E_1(F;T,d)+E_2(F;T,d)+E_3(F;T,d)+E_4(F;T,d),
\end{equation}
and
\begin{equation}\label{eq:El-def-a4}
E_\ell(F;T,d)
:=
E_1(F;T,d)-E_2(F;T,d)+E_3(F;T,d)+E_4(F;T,d).
\end{equation}
The operators $E_1, E_2, E_3,E_4$, and hence also $E_u, E_{\ell}$, are additive with respect to products:
$$
E_j(F G ; T, d)=E_j(F ; T, d)+E_j(G ; T, d), \quad j=1,2,3,4 .
$$
Therefore, since $\chi$ is non-trivial and
$$
\Xi_\chi(s)=A(\chi)^s \gamma_\chi(s)^2 L_\chi(s) L_{\bar{\chi}}(s)
$$
we have, for $\star \in\{u, \ell\}$,
\begin{equation}\label{eq:Estar-Xi-decomposition}
    E_{\star}\left(\Xi_\chi ; T, d\right)=E_{\star}\left(A(\chi)^s ; T, d\right)+2 E_{\star}\left(\gamma_\chi(s) ; T, d\right)+E_{\star}\left(L_\chi(s) ; T, d\right)+E_{\star}\left(L_{\bar{\chi}}(s) ; T, d\right).
\end{equation}
This is the decomposition that will be estimated term by term.
\subsection{The trivial character}

We now treat the case that $\chi=\chi_0$ is the trivial character. In this case
$L(s,\chi,L/K)=\zeta_K(s)$, and we do not use the auxiliary product
$\Xi_\chi(s)$. Instead, we work directly with the completed Dedekind
zeta-function
\[
        \xi_K(s)
        =
        s(s-1)d_K^{s/2}\gamma_K(s)\zeta_K(s).
\]
Thus the zero-counting argument is the same as above, but with $\Xi_\chi$
replaced by $\xi_K$. Let $\mathcal{Z}_K$ denote the multiset of non-trivial
zeros of $\zeta_K(s)$. Since $\xi_K(s)=\xi_K(1-s)$ and
$\xi_K(s)=\overline{\xi_K(\bar{s})}$, if
$\rho=\beta+i\gamma\in\mathcal{Z}_K$, then
$1-\bar{\rho}=1-\beta+i\gamma$ also belongs to $\mathcal{Z}_K$, with the
same multiplicity.\\
Applying the same argument in the Dedekind case gives
\begin{equation}\label{eq:NK-f-sum}
N_K(T)
=
\frac{1}{\pi}
\sum_{\rho=\beta+i\gamma\in\mathcal Z_K}
f\left(\frac12-\beta,T-\gamma;d\right)
+
E_1^K(\xi_K;T,d),
\end{equation}
where, for $j=1,2,3,4$, we define
\[
        E_j^K(F;T,d):=2E_j(F;T,d).
\]
Then
\[
E_u^K(F;T,d):=E_1^K(F;T,d)+E_2^K(F;T,d)+E_3^K(F;T,d)+E_4^K(F;T,d),
\]
and
\[
        E_\ell^K(F;T,d):=E_1^K(F;T,d)-E_2^K(F;T,d)+E_3^K(F;T,d)+E_4^K(F;T,d).
\]
With this notation, it follows that
\[
        E_\ell^K(\xi_K;T,d) \le N_K(T) \le E_u^K(\xi_K;T,d).
\]
In addition, recalling that $        \xi_K(s)=s(s-1)d_K^{s/2}\gamma_K(s)\zeta_K(s)$,
 we deduce
\begin{align}\label{eq:Estar-trivial-decomposition}
E_\star^K(\xi_K;T,d)
&=
E_\star^K(s(s-1);T,d)
+
E_\star^K(d_K^{s/2};T,d)+
E_\star^K(\gamma_K(s);T,d)
+
E_\star^K(\zeta_K(s);T,d)
\end{align}
 for $\star\in\{u,\ell\}$.
This is an analogue of \eqref{eq:Estar-Xi-decomposition} for the trivial character.
\begin{remark}
    Note that $E_u((s(s-1))^{\delta(\chi)})$ and $E_\ell((s(s-1))^{\delta(\chi)})$ are non-zero only if $\chi=\chi_0$ is the trivial character. In light of this, we will separate the case of trivial character in the following sections.
\end{remark}

\subsection{Estimating the terms}

We now estimate the terms appearing in
\eqref{eq:Estar-Xi-decomposition} and
\eqref{eq:Estar-trivial-decomposition}, following the strategy of
\cite{amberger2026estimatingnumberzerosdedekind}. The parameters
$d,a_1,a_2,a_3,a_4$ are kept variable for now and will be optimized later. Denote
\[
        m_\chi:=n_K\chi(1),
        \qquad
        \sigma_1:=\frac12+d,
        \qquad
        \sigma_2:=\frac12+2d.
\]
For $\sigma>1$, define
\[
        \mathcal L_1(\sigma):=-\frac{\zeta'}{\zeta}(\sigma),
        \qquad
        \mathcal L_2(\sigma):=
        \left(\frac{\zeta'}{\zeta}\right)'(\sigma),\qquad \mathcal L_3(\sigma):=-\left(\frac{\zeta'}{\zeta}\right)''(\sigma).
\]
Note that $\mathcal L_1(\sigma)$, $\mathcal L_2(\sigma)$, and $\mathcal L_3(\sigma)$ are positive for
$\sigma>1$.

\begin{lemma}\label{lem:conductor-term-artin}
For a non-trivial irreducible character $\chi$, one has
\[
E_u(A(\chi)^s;T,d)
=
\left(
\frac{T}{\pi}+\frac{d a_1}{4}
\right)\log A(\chi),
\]
and
\[
E_\ell(A(\chi)^s;T,d)
=
\left(
\frac{T}{\pi}-\frac{d a_1}{4}
\right)\log A(\chi).
\]
\end{lemma}

\begin{proof}
For $F(s)=A(\chi)^s$, we have
\[
        \log F(s)=s\log A(\chi),
        \qquad
        \frac{F'}{F}(s)=\log A(\chi),
        \qquad
        \left(\frac{F'}{F}\right)'(s)=0,\qquad \left(\frac{F'}{F}\right)''(s)=0.
\]
Substituting these identities into the definitions of $E_1,E_2,E_3,E_4$ gives
the result.
\end{proof}
For the trivial character we work directly with $\xi_K(s)$ rather than
with $\Xi_\chi(s)$. To match the normalization of the Dedekind zeta
zero-counting function, define
\[
        E_j^K(F;T,d):=2E_j(F;T,d),
        \qquad j=1,2,3,4
\]
and
\[
        E_u^K:=E_1^K+E_2^K+E_3^K+E_4^K,
        \qquad
        E_\ell^K:=E_1^K-E_2^K+E_3^K+E_4^K.
\]
Then
\[
        E_\ell^K(\xi_K;T,d)
        \le
        N_K(T)
        \le
        E_u^K(\xi_K;T,d).
\]

\begin{lemma}\label{lem:conductor-term-trivial}
When $\chi=\chi_0$, one has
\[
E_u^K(d_K^{s/2};T,d)
=
\left(
\frac{T}{\pi}+\frac{d a_1}{4}
\right)\log d_K,
\]
and
\[
E_\ell^K(d_K^{s/2};T,d)
=
\left(
\frac{T}{\pi}-\frac{d a_1}{4}
\right)\log d_K.
\]
\end{lemma}

\begin{proof}
For $F(s)=d_K^{s/2}$, one has
\[
        \log F(s)=\frac{s}{2}\log d_K,
        \qquad
        D_F(s)=\frac12\log d_K,
        \qquad
        D_F'(s)=D_F''(s)=0.
\]
Thus $E_3(F;T,d)=E_4(F;T,d)=0$. Moreover, the factor $1/2$ coming from
$d_K^{s/2}$ is cancelled by the normalization
$E_j^K=2E_j$. Substituting into the definitions of $E_u^K$ and $E_\ell^K$ gives the stated identities.
\end{proof}
\begin{lemma}\label{lem:polar-factor-trivial}
Let
\[
        x_+:=d+\frac12,
        \qquad
        x_-:=d-\frac12,
        \qquad
        y_+:=2d+\frac12,
        \qquad
        y_-:=2d-\frac12.
\]
For $F(s)=s(s-1)$, one has
\begin{align}\label{eq:E1K-polar-factor}
E_1^K(F;T,d)
&=
\frac{2}{\pi}
\left(
2\tan^{-1}\frac{T}{x_+}
+
2\tan^{-1}\frac{T}{x_-}
-
\tan^{-1}\frac{T}{y_+}
-
\tan^{-1}\frac{T}{y_-}
\right),
\end{align}
\begin{align}\label{eq:E2K-polar-factor}
E_2^K(F;T,d)
&=
\frac{d a_1}{2}
\left(
\frac{x_+}{x_+^2+T^2}
+
\frac{x_-}{x_-^2+T^2}
\right)
+
\frac{d^2a_2}{2}
\left(
\frac{x_+^2-T^2}{(x_+^2+T^2)^2}
+
\frac{x_-^2-T^2}{(x_-^2+T^2)^2}
\right),
\end{align}
\begin{align}\label{eq:E3K-polar-factor}
E_3^K(F;T,d)
&=
\frac{a_3}{2}
\left(
\frac{2x_+T}{(x_+^2+T^2)^2}
+
\frac{2x_-T}{(x_-^2+T^2)^2}
\right),
\end{align}
and
\begin{align}\label{eq:E4K-polar-factor-a4}
E_4^K(F;T,d)
&=
\frac{d^3a_4}{2}
\left(
\frac{3x_+^2T-T^3}{(x_+^2+T^2)^3}
+
\frac{3x_-^2T-T^3}{(x_-^2+T^2)^3}
\right).
\end{align}
Consequently,
\[
E_u^K(s(s-1);T,d)
=
E_1^K+E_2^K+E_3^K+E_4^K,
\]
and
\[
E_\ell^K(s(s-1);T,d)
=
E_1^K-E_2^K+E_3^K+E_4^K,
\]
where all terms are evaluated at $(s(s-1);T,d)$.
\end{lemma}

\begin{proof}
For $F(s)=s(s-1)$, we have
\[
        \frac{F'}{F}(s)=\frac{1}{s}+\frac{1}{s-1},
        \quad\text{and}\quad
        \left(\frac{F'}{F}\right)'(s)
        =
        -\frac{1}{s^2}
        -
        \frac{1}{(s-1)^2}.
\]
Substituting $s_0=1/2+d+iT$ into the definitions of $E_1^K,E_2^K,E_3^K,E_4^K$ gives the stated relations.
\end{proof}
We now adapt the Euler-product estimate of
\cite[Lemma 3.8]{amberger2026estimatingnumberzerosdedekind} to the four-parameter setting. The only new feature is the contribution of the operator $E_4$, which produces an additional third-order local term. For $q>1$ and $\phi\in\mathbb R$, define
\[
        w_1=w_1(q,\phi):=q^{-\sigma_1}e^{i\phi},
        \qquad
        w_2=w_2(q,\phi):=q^{-\sigma_2}e^{i\phi},
        \qquad
        L_q:=\log q.
\]
Let
\[
        A_j(q,\phi):=\Im\bigl(-\log(1-w_j)\bigr),
        \qquad j=1,2,
\]
where the logarithm is the analytic branch in the unit disc. Further set
\[
B_1(q,\phi):=-L_q\frac{w_1}{1-w_1},\qquad  B_2(q,\phi):=L_q^2\frac{w_1}{(1-w_1)^2},\qquad   B_3(q,\phi):=-L_q^3\frac{w_1(1+w_1)}{(1-w_1)^3}.
\]
Define the quantity $q_u(q,\phi)$ by
\begin{align}\label{eq:qu-four-param}
\frac{2}{\pi}
\left(
2A_1(q,\phi)-A_2(q,\phi)
\right)
+
\frac{d a_1}{2}\Re B_1(q,\phi)
-
\frac{d^2a_2}{2}\Re B_2(q,\phi)
+
\frac{a_3}{2}\Im B_2(q,\phi)
-
\frac{d^3a_4}{4}\Im B_3(q,\phi),
\end{align}
which is the four-parameter analogue of
Amberger's function $q_1(q,\phi)$. More precisely,
\[
        q_u(q,\phi)
        =
        q_1(q,\phi)
        -
        \frac{d^3a_4}{4}\Im B_3(q,\phi),
\]
where the extra term is the new contribution coming from $E_4$. Finally, define
\begin{equation}\label{eq:CE-four-param}
\mathcal C_E(d,a_1,a_2,a_3,a_4)
:=
\sum_p
\max\left\{
0,\,
\sup_{\substack{m\ge1\\ \phi\in[0,2\pi]}}
q_u(p^m,\phi)
\right\},
\end{equation}
where the sum is over primes.

\begin{lemma}\label{lem:four-param-euler-contribution}
Let $\chi$ be a non-trivial irreducible character of $\Gal(L/K)$, and put
$m_\chi=n_K\chi(1)$. Then
\begin{equation}\label{eq:artin-euler-upper-four-param}
E_u(L_\chi;T,d)+E_u(L_{\bar\chi};T,d)
\le
m_\chi\,
\mathcal C_E(d,a_1,a_2,a_3,a_4),
\end{equation}
and
\begin{equation}\label{eq:artin-euler-lower-four-param}
E_\ell(L_\chi;T,d)+E_\ell(L_{\bar\chi};T,d)
\ge
-
m_\chi\,
\mathcal C_E(d,a_1,a_2,a_3,a_4).
\end{equation}
\end{lemma}

\begin{proof}
For $\Re s>1$, each local Artin $L$-function can be written as
\[
        L_{\mathfrak p}(s,\chi)
        =
        \prod_{j=1}^{\chi(1)}
        \left(1-\lambda_{j,\mathfrak p}(\mathrm N\mathfrak p)^{-s}\right)^{-1},
\]
where $|\lambda_{j,\mathfrak p}|\le1$, after adjoining zero eigenvalues if
necessary. It is therefore enough to consider a single local factor
\[
        F_z(s):=(1-zq^{-s})^{-1},
        \qquad |z|\le1,
        \qquad q=\mathrm N\mathfrak p.
\]
The contribution of $F_z$ to $E_u$ is a real linear combination of real and
imaginary parts of functions holomorphic in $z$ on the disc $|z|<1$ and
continuous on $|z|\le1$. Hence, by the maximum principle for harmonic
functions, its maximum over $|z|\le1$ is attained on the unit circle. Write $z=e^{i\theta}$ and let $\phi=\theta-T\log q$. Then
\[
        zq^{-(\frac12+d+iT)}
        =
        q^{-\sigma_1}e^{i\phi}
        =
        w_1,
        \qquad
        zq^{-(\frac12+2d+iT)}
        =
        q^{-\sigma_2}e^{i\phi}
        =
        w_2.
\]
Moreover,
\[
        \frac{F_z'}{F_z}\left(\frac12+d+iT\right)
        =
        -(\log q)\frac{w_1}{1-w_1}
        =
        B_1(q,\phi),
\]
\[
        \left(\frac{F_z'}{F_z}\right)'
        \left(\frac12+d+iT\right)
        =
        (\log q)^2\frac{w_1}{(1-w_1)^2}
        =
        B_2(q,\phi),
\]
and
\[
        \left(\frac{F_z'}{F_z}\right)''
        \left(\frac12+d+iT\right)
        =
        -(\log q)^3
        \frac{w_1(1+w_1)}{(1-w_1)^3}
        =
        B_3(q,\phi).
\]
Substituting these identities into the definition $  E_u=E_1+E_2+E_3+E_4$ gives
\[
        E_u(F_z;T,d)=\frac12 q_u(q,\phi).
\]
Hence the paired local contribution from $L_\chi(s)L_{\bar\chi}(s)$ is
bounded above by $ q_u(q,\phi)$. For each prime ideal $\mathfrak p$, summing over the at most $\chi(1)$ local
roots gives
\[
        \chi(1)
        \sup_{\phi\in[0,2\pi]}
        q_u(\mathrm N\mathfrak p,\phi).
\]
Therefore
\[
E_u(L_\chi;T,d)+E_u(L_{\bar\chi};T,d)
\le
\chi(1)
\sum_{\mathfrak p\subset\mathcal O_K}
\sup_{\phi\in[0,2\pi]}
q_u(\mathrm N\mathfrak p,\phi).
\]
For a rational prime $p$, every prime ideal of $K$ above $p$ has norm $p^m$
for some $m\ge1$, and the number of prime ideals above $p$ is at most
$n_K$. Hence
\[
\sum_{\mathfrak p\subset\mathcal O_K}
\sup_{\phi}
q_u(\mathrm N\mathfrak p,\phi)
\le
n_K
\sum_p
\max\left\{
0,\,
\sup_{\substack{m\ge1\\ \phi\in[0,2\pi]}}
q_u(p^m,\phi)
\right\}.
\]
By the definition of $\mathcal C_E$, this proves
\eqref{eq:artin-euler-upper-four-param}.\\
For the lower bound, let $q_\ell(q,\phi)$ denote the contribution
obtained from one prime by replacing $E_u$ with $E_\ell$. Since the terms coming from $E_1$,
$\Im B_2$, and $\Im B_3$ are odd in $\phi$, whereas $\Re B_1$ and $\Re B_2$
are even in $\phi$, one has
\[
        q_\ell(q,\phi)
        =
        -q_u(q,-\phi).
\]
Therefore
\[
        \inf_{\phi\in[0,2\pi]}q_\ell(q,\phi)
        =
        -
        \sup_{\phi\in[0,2\pi]}q_u(q,\phi).
\]
The same summation argument gives
\eqref{eq:artin-euler-lower-four-param}.
\end{proof}
\begin{lemma}\label{cor:dedekind-euler-four-param}
With the notation of Lemma \ref{lem:four-param-euler-contribution}, one has, for the trivial character,
\[
        E_u^K(\zeta_K;T,d)
        \le
        n_K\,
        \mathcal C_E(d,a_1,a_2,a_3,a_4),
\]
and
\[
        E_\ell^K(\zeta_K;T,d)
        \ge
        -n_K\,
        \mathcal C_E(d,a_1,a_2,a_3,a_4).
\]
\end{lemma}

\begin{proof}
One applies the same majorization used in the proof of Lemma \ref{lem:four-param-euler-contribution} to the Euler product of $\zeta_K(s)$, using the normalization $E_j^K=2E_j,$ for $j=1,2,3,4.$
\end{proof}
For completeness, we also mention two results which could be used instead of Lemma \ref{lem:four-param-euler-contribution} and Lemma \ref{cor:dedekind-euler-four-param}. The proofs of Lemma \ref{lem:artin-euler-term} and Lemma \ref{lem:dedekind-euler-term} are much simpler, but they lead to less sharp estimates.
\begin{lemma}\label{lem:artin-euler-term}
Let $\chi$ be non-trivial. For $\star\in\{u,\ell\}$, the Euler-product
contribution satisfies
\[
        E_u(L_\chi;T,d)+E_u(L_{\bar{\chi}};T,d)
        \le
        m_\chi \mathcal R_L(d,a_1,a_2,a_3,a_4),
\]
and
\[
        E_\ell(L_\chi;T,d)+E_\ell(L_{\bar{\chi}};T,d)
        \ge
        -m_\chi \mathcal R_L(d,a_1,a_2,a_3,a_4),
\]
where
\begin{align}\label{eq:RL-def}
\begin{split}
&\mathcal R_L(d,a_1,a_2,a_3,a_4)\\
&:=
\frac{2}{\pi}
\left(
2\log\zeta(\sigma_1)+\log\zeta(\sigma_2)
\right)
+
\frac{d a_1}{2}\mathcal L_1(\sigma_1) +
\frac{d^2|a_2|+|a_3|}{2}\mathcal L_2(\sigma_1)
+
\frac{|d^3a_4|}{4}\mathcal L_3(\sigma_1).
\end{split}
\end{align}
\end{lemma}

\begin{proof}
For $\sigma>1$, Corollary \ref{cor:artin-log-derivative-bounds} implies
\[ |\log L(s,\chi,L/K)|\le m_\chi\log\zeta(\sigma)
\quad \text{and}
\quad
\left| \frac{d^k}{dx^k} \left(\frac{L'}{L}\right)(s,\chi,L/K) \right| \le  m_\chi\mathcal L_{k+1}(\sigma)
\]
for $k\ge 0$. The same bounds hold for $L_{\bar{\chi}}$. Applying these estimates at
$\sigma_1=1/2+d$ and $\sigma_2=1/2+2d$ to the definitions of
$E_1,E_2,E_3,E_4$ gives the result.
\end{proof}
\begin{lemma}\label{lem:dedekind-euler-term}
For the trivial character, one has
\[ E_u^K(\zeta_K;T,d)   \le  n_K \mathcal R_L(d,a_1,a_2,a_3,a_4),
\]
and
\[ E_\ell^K(\zeta_K;T,d)  \ge  -n_K \mathcal R_L(d,a_1,a_2,a_3,a_4).
\]
\end{lemma}

\begin{proof}
The proof is the same as that of Lemma \ref{lem:artin-euler-term}, with
$\chi=\chi_0$ and $L(s,\chi,L/K)=\zeta_K(s)$, using the normalization
$E_j^K=2E_j$.
\end{proof}
\subsection{The gamma-factor contribution}
For $\varepsilon\in\{0,1\}$, let
\[
        \Gamma_\varepsilon(s)
        :=
        \pi^{-(s+\varepsilon)/2}
        \Gamma\left(\frac{s+\varepsilon}{2}\right).
\]
Thus, \eqref{def:gammachi} becomes
\[
        \gamma_\chi(s) = \Gamma_0(s)^{c_0(\chi)} \Gamma_1(s)^{c_1(\chi)}.
\]
Define
\[
        \mathfrak E_j(F;T,d):=2E_j(F;T,d),
        \qquad j=1,2,3,4,
\]
and
\[
        \mathfrak E_u:=\mathfrak E_1+\mathfrak E_2+\mathfrak E_3+\mathfrak E_4,
        \qquad
        \mathfrak E_\ell:=\mathfrak E_1-\mathfrak E_2+\mathfrak E_3+\mathfrak E_4.
\]
For $\varepsilon\in\{0,1\}$, define
\begin{align}\label{eq:gamma-remainder-upper-a4}
\mathcal G_{\varepsilon,u}(T;d,a_1,a_2,a_3,a_4)
&:=
\mathfrak E_u(\Gamma_\varepsilon;T,d)
-
\frac{T}{\pi}\log\left(\frac{T}{2\pi e}\right)
-
\frac{d a_1}{4}\log\left(\frac{T}{2\pi}\right),
\end{align}
and
\begin{align}\label{eq:gamma-remainder-lower-a4}
\mathcal G_{\varepsilon,\ell}(T;d,a_1,a_2,a_3,a_4)
&:=
\mathfrak E_\ell(\Gamma_\varepsilon;T,d)
-
\frac{T}{\pi}\log\left(\frac{T}{2\pi e}\right)
+
\frac{d a_1}{4}\log\left(\frac{T}{2\pi}\right).
\end{align}
Finally set
\[
        \mathcal G_u(d,a_1,a_2,a_3,a_4)
        :=
        \max_{\varepsilon\in\{0,1\}}
        \sup_{T\ge 1}
        \mathcal G_{\varepsilon,u}(T;d,a_1,a_2,a_3,a_4),
\]
and
\[
        \mathcal G_\ell(d,a_1,a_2,a_3,a_4)
        :=
        \min_{\varepsilon\in\{0,1\}}
        \inf_{T\ge 1}
        \mathcal G_{\varepsilon,\ell}(T;d,a_1,a_2,a_3,a_4).
\]
\begin{remark}
The quantities $\mathcal G_u$ and $\mathcal G_\ell$ are introduced here to keep the dependence on $d,a_1,a_2,a_3,a_4$ explicit, as these parameters will be optimized later. They play the same role as the uniform numerical bounds for the remainder terms
$U_{j,1}(T)+U_{j,2}(T)$ and $L_{j,1}(T)+L_{j,2}(T)$ in the proof of \cite[Lemma 3.7]{amberger2026estimatingnumberzerosdedekind}.
\end{remark}
\begin{lemma}\label{lem:gamma-artin}
For a non-trivial irreducible character $\chi$, one has
\begin{align}\label{eq:gamma-artin-upper-a4}
2E_u(\gamma_\chi;T,d)
&\le
m_\chi
\left[
\frac{T}{\pi}\log\left(\frac{T}{2\pi e}\right)
+
\frac{d a_1}{4}\log\left(\frac{T}{2\pi}\right)
+
\mathcal G_u(d,a_1,a_2,a_3,a_4)
\right],
\end{align}
and
\begin{align}\label{eq:gamma-artin-lower-a4}
2E_\ell(\gamma_\chi;T,d)
&\ge
m_\chi
\left[
\frac{T}{\pi}\log\left(\frac{T}{2\pi e}\right)
-
\frac{d a_1}{4}\log\left(\frac{T}{2\pi}\right)
+
\mathcal G_\ell(d,a_1,a_2,a_3,a_4)
\right].
\end{align}
\end{lemma}

\begin{proof}
Since
\[
        \gamma_\chi(s)
        =
        \Gamma_0(s)^{c_0(\chi)}
        \Gamma_1(s)^{c_1(\chi)}
\]
and $c_0(\chi)+c_1(\chi)=m_\chi$, additivity gives
\[
        2E_\star(\gamma_\chi;T,d)
        =
        c_0(\chi)\mathfrak E_\star(\Gamma_0;T,d)
        +
        c_1(\chi)\mathfrak E_\star(\Gamma_1;T,d),
        \quad\text{for } 
        \star\in\{u,\ell\}.
\]
The result follows from the definitions of
$\mathcal G_u$ and $\mathcal G_\ell$.
\end{proof}

\begin{lemma}\label{lem:gamma-trivial}
For the trivial character, one has
\begin{align}\label{eq:gamma-trivial-upper}
E_u^K(\gamma_K;T,d)
&\le
n_K
\left[
\frac{T}{\pi}\log\left(\frac{T}{2\pi e}\right)
+
\frac{d a_1}{4}\log\left(\frac{T}{2\pi}\right)
+
\mathcal G_u(d,a_1,a_2,a_3,a_4)
\right],
\end{align}
and
\begin{align}\label{eq:gamma-trivial-lower}
E_\ell^K(\gamma_K;T,d)
&\ge
n_K
\left[
\frac{T}{\pi}\log\left(\frac{T}{2\pi e}\right)
-
\frac{d a_1}{4}\log\left(\frac{T}{2\pi}\right)
+
\mathcal G_\ell(d,a_1,a_2,a_3,a_4)
\right].
\end{align}
\end{lemma}

\begin{proof}
For the Dedekind zeta-function,
\[
        \gamma_K(s)
        =
        \Gamma_0(s)^{r_1+r_2}
        \Gamma_1(s)^{r_2},
\]
and $(r_1+r_2)+r_2=n_K$. The proof is therefore identical to that of
Lemma \ref{lem:gamma-artin}.
\end{proof}

\begin{remark}  (i)
One may ask whether the gamma-factor estimates used in
\cite[Section\,3]{BMOR20} could improve the final constants in
Theorems \ref{main-thm} and \ref{Dedekind:thm}. Such an improvement would only affect the constants \(\mathcal G_u\) and \(\mathcal G_\ell\) appearing in Lemmas \ref{lem:gamma-artin} and \ref{lem:gamma-trivial}, and not the constant \(\mathcal C_E\). Since \eqref{const:C1C2} yields
\[
        C_2
        =
        \max\{\mathcal C_E+\mathcal G_u,\,
        \mathcal C_E-\mathcal G_\ell\},
\]
any improvement coming only from the gamma factor is limited by the size of
\(\mathcal C_E\). In particular, for the admissible tuple \eqref{eq:tuple-small-C1} used in this paper, we have $  \mathcal C_E\le 9.168629168235.$ Thus even in the ideal situation where
\(\mathcal G_u=\mathcal G_\ell=0\), one would still obtain $C_2\approx \mathcal C_E\approx 9.169,$ which remains slightly larger than the value \(C_2=8.161\) found in \cite{amberger2026estimatingnumberzerosdedekind}.\\
\noindent (ii) For completeness, we also tested the gamma-factor contribution using the
approach of \cite[Section\,3]{BMOR20}. This gives essentially the same
bounds as those obtained from Lemmas \ref{lem:gamma-artin} and
\ref{lem:gamma-trivial}; see Table
\ref{tab:certified-intermediate-constants}. More precisely, we obtain $ \mathcal G_u\le 0.315174,$ and $\mathcal G_\ell\ge -0.250001.$
Therefore, at least for the present choice of parameters, the estimates used to bound the gamma factor in \cite[Section\,3]{BMOR20} do not appear to lead to a noticeable improvement in
the final constant \(C_2\).
\end{remark}
\subsection{Combining the estimates}
For a non-trivial irreducible character $\chi$, define
\[
        M_\chi(T)
        :=
        \frac{T}{\pi}
        \log\left(
        A(\chi)
        \left(\frac{T}{2\pi e}\right)^{m_\chi}
        \right),
        \qquad
        m_\chi=n_K\chi(1).
\]

\begin{proposition}\label{prop:nontrivial-artin-two-sided}
Let $\chi$ be non-trivial and assume Artin's holomorphy conjecture for
$\chi$. Let $(d,a_1,a_2,a_3,a_4)$ be admissible. Then, for $T\ge1$,
\[
N(T,\chi)
\le
M_\chi(T)
+
\frac{d a_1}{4}
\left(
\log A(\chi)
+
m_\chi\log\left(\frac{T}{2\pi}\right)
\right)
+
m_\chi
\left(
\mathcal G_u+\mathcal C_E
\right),
\]
and
\[
N(T,\chi)
\ge
M_\chi(T)
-
\frac{d a_1}{4}
\left(
\log A(\chi)
+
m_\chi\log\left(\frac{T}{2\pi}\right)
\right)
+
m_\chi
\left(
\mathcal G_\ell-\mathcal C_E
\right).
\]
Here, 
\[
        \mathcal C_E=\mathcal C_E(d,a_1,a_2,a_3,a_4),
        \quad
        \mathcal G_\star=\mathcal G_\star(d,a_1,a_2,a_3,a_4),\quad\text{for }   \star\in\{u,\ell\},
\]
and $\mathcal G_\ell$ is taken with its sign, as in the lower gamma-factor estimate.
\end{proposition}

\begin{proof}
By the decomposition
\[
E_\star(\Xi_\chi;T,d)
=
E_\star(A(\chi)^s;T,d)
+
2E_\star(\gamma_\chi;T,d)
+
E_\star(L_\chi;T,d)
+
E_\star(L_{\bar{\chi}};T,d),
\]
for $\star\in\{u,\ell\}$, the result follows from
\[
        E_\ell(\Xi_\chi;T,d)
        \le
        N(T,\chi)
        \le
        E_u(\Xi_\chi;T,d),
\]
together with Lemmas \ref{lem:conductor-term-artin},
\ref{lem:four-param-euler-contribution}, and \ref{lem:gamma-artin}.
\end{proof}
For the trivial character, we define
\[
        M_K(T)
        :=
        \frac{T}{\pi}
        \log\left(
        d_K
        \left(\frac{T}{2\pi e}\right)^{n_K}
        \right).
\]

\begin{proposition}\label{prop:trivial-two-sided}
For the trivial character $\chi=\chi_0$, one has, for $T\ge1$,
\[
N_K(T)
\le
M_K(T)
+
\frac{d a_1}{4}
\left(
\log d_K
+
n_K\log\left(\frac{T}{2\pi}\right)
\right)
+
n_K
\left(
\mathcal G_u+\mathcal C_E
\right)
+
E_u^K(s(s-1);T,d),
\]
and
\[
N_K(T)
\ge
M_K(T)
-
\frac{d a_1}{4}
\left(
\log d_K
+
n_K\log\left(\frac{T}{2\pi}\right)
\right)
+
n_K
\left(
\mathcal G_\ell-\mathcal C_E
\right)
+
E_\ell^K(s(s-1);T,d).
\]
\end{proposition}

\begin{proof}
By the decomposition
\[
E_\star^K(\xi_K;T,d)
=
E_\star^K(s(s-1);T,d)
+
E_\star^K(d_K^{s/2};T,d)
+
E_\star^K(\gamma_K;T,d)
+
E_\star^K(\zeta_K;T,d),
\]
for $\star\in\{u,\ell\}$, the result follows from
\[
        E_\ell^K(\xi_K;T,d)
        \le
        N_K(T)
        \le
        E_u^K(\xi_K;T,d),
\]
together with Lemmas \ref{lem:conductor-term-trivial},
\ref{lem:polar-factor-trivial}, \ref{cor:dedekind-euler-four-param}, and
\ref{lem:gamma-trivial}.
\end{proof}
Combining the upper and lower bounds gives the following symmetric form.
For non-trivial $\chi$,
\[
\left|
N(T,\chi)-M_\chi(T)
\right|
\le
C_1
\left(
\log A(\chi)
+
m_\chi\log\left(\frac{T}{2\pi}\right)
\right)
+
C_2m_\chi,
\]
where
\begin{equation}\label{const:C1C2}
      C_1=\frac{d a_1}{4},\qquad 
        C_2
        =
        \max\left\{
        \mathcal C_E+\mathcal G_u,\,
        \mathcal C_E-\mathcal G_\ell
        \right\}.
\end{equation}
Here again $\mathcal G_\ell$ is the signed constant appearing in the lower
gamma-factor estimate.\\
For the trivial character,
\[
\left|
N_K(T)-M_K(T)
\right|
\le
C_1
\left(
\log d_K
+
n_K\log\left(\frac{T}{2\pi}\right)
\right)
+
C_2n_K
+
C_3,
\]
where $C_1$ and $C_2$ are as above, and
\[
        C_3
        :=
        \sup_{T\ge1}
        \max\left\{
        E_u^K(s(s-1);T,d),
        -
        E_\ell^K(s(s-1);T,d)
        \right\}.
\]
Equivalently,
\[
E_u^K(s(s-1);T,d)
=
E_1^K(s(s-1);T,d)
+
E_2^K(s(s-1);T,d)
+
E_3^K(s(s-1);T,d)
+
E_4^K(s(s-1);T,d),
\]
and
\[
E_\ell^K(s(s-1);T,d)
=
E_1^K(s(s-1);T,d)
-
E_2^K(s(s-1);T,d)
+
E_3^K(s(s-1);T,d)
+
E_4^K(s(s-1);T,d).
\]
\subsection{Admissible tuples}

We use two admissible tuples. The first one gives a smaller leading
constant \(C_1\), while the second one gives a smaller value of \(C_2\), with
\(C_1\) and \(C_3\) matching the corresponding values in Amberger's estimate.\\
The first tuple is
\begin{equation}\label{eq:tuple-small-C1}
        d=0.713,\qquad
        a_1=1.061,\qquad
        a_2=0.940,\qquad
        a_3=0.315,\qquad
        a_4=-0.300.
\end{equation}
The second tuple is
\begin{equation}\label{eq:tuple-small-C2}
        d=0.722382,\qquad
        a_1=1.074223,\qquad
        a_2=0.925936,\qquad
        a_3=0.327241,\qquad
        a_4=-0.138682.
\end{equation}
The certified intermediate constants for these tuples are listed in
Table~\ref{tab:certified-intermediate-constants}.
\begin{table}[ht]
\centering
\renewcommand{\arraystretch}{1.35}
\begin{tabular}{|>{\centering\arraybackslash}p{0.18\textwidth}|
                >{\centering\arraybackslash}p{0.34\textwidth}|
                >{\centering\arraybackslash}p{0.34\textwidth}|}
\hline
Quantity
& Tuple \eqref{eq:tuple-small-C1}
& Tuple \eqref{eq:tuple-small-C2} \\
\hline
$\mathcal C_E$
& $\mathcal C_E\le 9.168629168235$
& $\mathcal C_E\le 7.590620795745$ \\
\hline
$\mathcal G_u$
& $\mathcal G_u\le 0.315173228521$
& $\mathcal G_u\le 0.316$ \\
\hline
$\mathcal G_\ell$
& $\mathcal G_\ell\ge -0.250001$
& $\mathcal G_\ell\ge -0.251$ \\
\hline
\end{tabular}
\caption{Certified intermediate constants for the two admissible tuples.}
\label{tab:certified-intermediate-constants}
\end{table}
The resulting constants in the final zero-counting estimates are listed in
Table~\ref{tab:final-four-param-constants}.
\begin{table}[ht]
\centering
\renewcommand{\arraystretch}{1.35}
\begin{tabular}{|>{\centering\arraybackslash}p{0.22\textwidth}|
                >{\centering\arraybackslash}p{0.22\textwidth}|
                >{\centering\arraybackslash}p{0.22\textwidth}|
                >{\centering\arraybackslash}p{0.22\textwidth}|}
\hline
Tuple
& Case
& $(C_1,C_2,C_3)$
& Improved constant \\
\hline
\eqref{eq:tuple-small-C1}
& Non-trivial Artin
& $(0.1892,\,9.484,\,0)$
& Smaller \(C_1\) \\
\hline
\eqref{eq:tuple-small-C1}
& Trivial/Dedekind
& $(0.1892,\,9.484,\,2.007)$
& Smaller \(C_1\) \\
\hline
\eqref{eq:tuple-small-C2}
& Non-trivial Artin
& $(0.194,\,7.907,\,0)$
& Smaller \(C_2\) \\
\hline
\eqref{eq:tuple-small-C2}
& Trivial/Dedekind
& $(0.194,\,7.907,\,2.001)$
& Smaller \(C_2\) \\
\hline
\end{tabular}
\caption{Final constants obtained from the two admissible tuples.}
\label{tab:final-four-param-constants}
\end{table}

\section{Final remarks on possible improvements}

We briefly indicate a possible extension of the four-parameter method. One
can introduce two further parameters $a_5$ and $a_6$ by adding two
additional terms to the upper bound for $f(b,t;d)$. More
precisely, one adds the terms
\[
d^4a_5\{R_4(d+b,t)+R_4(d-b,t)\}
+
d^4a_6\{I_4(d+b,t)+I_4(d-b,t)\},
\]
where
\[
R_4(x,t)
=
\Re\frac{1}{(x+it)^4}
=
\frac{x^4-6x^2t^2+t^4}{(x^2+t^2)^4},
\]
and
\[
I_4(x,t)
=
-\Im\frac{1}{(x+it)^4}
=
\frac{4xt(x^2-t^2)}{(x^2+t^2)^4}.
\]
The rest of the argument is formally the same as in the four-parameter case.
The two new terms give rise to two additional operators involving the third
derivative of the logarithmic derivative, and the Euler and gamma
contributions must then be recomputed with the corresponding extra local
terms. For example, with the tuple
\[
\begin{aligned}
d&=0.720,        &
a_1&=0.98120,   &
a_2&=0.95435,   \\
a_3&=0.46130,   &
a_4&=-0.29676,  &
a_5&=0.06448,   &
a_6&=0.14232,
\end{aligned}
\]
we obtain the constants in Table~\ref{tab:six-param-final-constants}. These values suggest that the leading constant $C_1$ can be reduced
substantially by adding further parameters. However, in this particular
six-parameter example, the lower-order constants $C_2$ and $C_3$ become
larger than in the four-parameter version. Thus the four-parameter method
appears to provide a better overall balance among the constants, while the
six-parameter variant may be useful if one is primarily interested in
minimizing the leading coefficient.
\begin{table}[ht]
\centering
\renewcommand{\arraystretch}{1.35}
\begin{tabular}{|c|c|c|c|}
\hline
Case & $C_1$ & $C_2$ & $C_3$ \\
\hline
Non-trivial Artin case
& $0.1767$
& $15.805$
& $0$ \\
\hline
Trivial/Dedekind case
& $0.1767$
& $15.805$
& $2.033$ \\
\hline
\end{tabular}
\caption{Constants obtained from a preliminary six-parameter computation.}
\label{tab:six-param-final-constants}
\end{table}

We emphasize that this six-parameter tuple is only preliminary. Our
computation was based on a direct search for admissible parameters followed by rigorous certification, rather than on a convex-optimization formulation.
For fixed \(d\), however, the admissibility inequalities are linear in the
parameters \(a_1,\ldots,a_6\), so the admissible region is convex. It is
therefore plausible that more systematic fixed-\(d\) searches, for instance, based on discretized convex optimization, could lead to better six-parameter tuples than the preliminary example above. In particular, one might be able to find tuples with nonzero \(a_5\) and \(a_6\) that reduce
\(C_1\) without increasing \(C_2\) as much as in our preliminary computation. 

For the four-parameter family, there is a simple obstruction to making \(C_1\)
arbitrarily small. Indeed, there are some limiting constraints obtained from the admissible inequality when \(|t|\to0\) and \(|t|\to\infty\) which imply
\[
        C_1=\frac{d a_1}{4}>\frac18.
\]
This particular obstruction disappears once further even terms, such as the \(R_4\)-term above, are introduced. This suggests that adding more
parameters may reduce the leading constant further. Some preliminary
computations with eight parameters show that \(C_1\) can be pushed slightly further, to about \(0.1729\), but the corresponding \(C_2\) becomes larger, around \(19.4\). Thus, the numerical evidence suggests that higher-parameter variants of our approach can reduce the leading constant, but at the cost of a rapid growth in the lower-order constants.

It would be interesting to understand the limiting behaviour of this optimization problem. In particular, one may ask whether the infimum of \(C_1\), taken over all versions of this method with finitely many additional higher-order terms, is equal to \(0\). Nevertheless, we shall not pursue this question and the above-mentioned optimization problems further here, and reserve them as a future research direction.

\section*{Acknowledgments}

The authors thank Andrew Fiori,  Nathan Ng, and Tim Trudgian for their comments and suggestions.

\printbibliography

@article{KN12, 
    AUTHOR = {Kadiri, H. and Ng, N.},
     TITLE = {Explicit zero density theorems for {D}edekind zeta functions},
   JOURNAL = {J. Number Theory},
  FJOURNAL = {Journal of Number Theory},
    VOLUME = {132},
      YEAR = {2012},
    NUMBER = {4},
     PAGES = {748--775},
      ISSN = {0022-314X},
   MRCLASS = {11R42 (11M41)},
  MRNUMBER = {2887617},
MRREVIEWER = {St\'{e}phane R. Louboutin},
       DOI = {10.1016/j.jnt.2011.09.002},
       URL = {https://doi.org/10.1016/j.jnt.2011.09.002},
}

@article{McC84,
    AUTHOR = {McCurley, K. S.},
     TITLE = {Explicit estimates for the error term in the prime number
              theorem for arithmetic progressions},
   JOURNAL = {Math. Comp.},
  FJOURNAL = {Mathematics of Computation},
    VOLUME = {42},
      YEAR = {1984},
    NUMBER = {165},
     PAGES = {265--285},
      ISSN = {0025-5718},
   MRCLASS = {11N13 (11-04 11Y35)},
  MRNUMBER = {726004},
MRREVIEWER = {Matti Jutila},
       DOI = {10.2307/2007579},
       URL = {https://doi.org/10.2307/2007579},
}

@article{Ro41,
    AUTHOR = {Rosser, B.},
     TITLE = {Explicit bounds for some functions of prime numbers},
   JOURNAL = {Amer. J. Math.},
  FJOURNAL = {American Journal of Mathematics},
    VOLUME = {63},
      YEAR = {1941},
     PAGES = {211--232},
      ISSN = {0002-9327},
   MRCLASS = {10.0X},
  MRNUMBER = {3018},
MRREVIEWER = {R. D. James},
       DOI = {10.2307/2371291},
       URL = {https://doi.org/10.2307/2371291},
}

@article{Tr15,
    AUTHOR = {Trudgian, T. S.},
     TITLE = {An improved upper bound for the error in the zero-counting
              formulae for {D}irichlet {$L$}-functions and {D}edekind
              zeta-functions},
   JOURNAL = {Math. Comp.},
  FJOURNAL = {Mathematics of Computation},
    VOLUME = {84},
      YEAR = {2015},
    NUMBER = {293},
     PAGES = {1439--1450},
      ISSN = {0025-5718},
   MRCLASS = {11M06 (11M26 11R42)},
  MRNUMBER = {3315515},
MRREVIEWER = {Caroline L. Turnage-Butterbaugh},
       DOI = {10.1090/S0025-5718-2014-02898-6},
       URL = {https://doi.org/10.1090/S0025-5718-2014-02898-6},
}

@article{Trudgian20121053,
	author = {Trudgian,  T. S.},
	title = {An improved upper bound for the argument of the Riemann zeta-function on the critical line},
	year = {2012},
	journal = {Math. Comp.},
	volume = {81},
	number = {278},
	pages = {1053--1061},
}

@article{Tr14-2,
    AUTHOR = {Trudgian, T. S.},
     TITLE = {An improved upper bound for the argument of the {R}iemann
              zeta-function on the critical line {II}},
   JOURNAL = {J. Number Theory},
  FJOURNAL = {Journal of Number Theory},
    VOLUME = {134},
      YEAR = {2014},
     PAGES = {280--292},
      ISSN = {0022-314X},
   MRCLASS = {11M06 (11M26)},
  MRNUMBER = {3111568},
MRREVIEWER = {Haseo Ki},
       DOI = {10.1016/j.jnt.2013.07.017},
       URL = {https://doi.org/10.1016/j.jnt.2013.07.017},
}

@article{Ba18,
AUTHOR={Backlund, R. J.},
     TITLE = {\"{U}ber die {N}ullstellen der {R}iemannschen {Z}etafunktion},
   JOURNAL = {Acta Math.},
  FJOURNAL = {Acta Mathematica},
    VOLUME = {41},
      YEAR = {1916},
    NUMBER = {1},
     PAGES = {345--375},
      ISSN = {0001-5962},
   MRCLASS = {DML},
  MRNUMBER = {1555156},
       DOI = {10.1007/BF02422950},
       URL = {https://doi.org/10.1007/BF02422950},
}

@article{Gr13, 
AUTHOR={Grossmann, J.},
TITLE={\"Uber die Nullstellen der Riemannschen Zeta-Funktion und der Dirichletschen $L$-Funktionen},
note={PhD thesis, Georg-August-Universit\"at G\"ottingen},
year={1913},
}

@article{vMo05,
    AUTHOR = {Von Mangoldt, H. C. F.},
     TITLE = {Zur {V}erteilung der {N}ullstellen der {R}iemannschen
              {F}unktion {$\xi(t)$}},
   JOURNAL = {Math. Ann.},
  FJOURNAL = {Mathematische Annalen},
    VOLUME = {60},
      YEAR = {1905},
    NUMBER = {1},
     PAGES = {1--19},
      ISSN = {0025-5831},
   MRCLASS = {DML},
  MRNUMBER = {1511287},
       DOI = {10.1007/BF01447494},
       URL = {https://doi.org/10.1007/BF01447494},
}

@article{HASANALIZADE2022219,
title = {Counting zeros of the Riemann zeta function},
journal = {J. Number Theory},
volume = {235},
pages = {219-241},
year = {2022},
author = {Elchin Hasanalizade and Quanli Shen and Peng-Jie Wong},
}

@article{HSW21-Dedekind,
  title={Counting zeros of Dedekind zeta functions},
  author={Elchin Hasanalizade and Quanli Shen and Peng-Jie Wong},
  journal={Math. Comp.},
  year={2021},
  volume={91},
  pages={277-293},
}

@article{BMOR20,
	title = {Counting zeros of {Dirichlet} $L$-functions},
	volume = {90},
	number = {329},
	journal = {Math. Comp.},
	author = {Michael A. Bennett  and Greg Martin and Kevin O’Bryant  and Andrew Rechnitzer},
	year = {2021},
	pages = {1455--1482},
}

@article{PLATT2015842,
title = {An improved explicit bound on $|\zeta(1/2+it)|$},
journal = {J. Number Theory},
volume = {147},
pages = {842-851},
year = {2015},
author = {Platt, D. J. and Trudgian,  T. S.}
}

@inproceedings{martinet1977,
  author    = {Martinet, J.},
  title     = {Character theory and Artin L-functions},
  booktitle = {Algebraic Number Fields},
  editor    = {Fr\"ohlich, A.},
  year      = {1977},
  publisher = {Academic Press},
  address   = {London},
  pages     = {1--87},
}

@article{Wong2018,
author = {Wong, Peng-Jie},
title = {Applications of group theory to conjectures of Artin and Langlands},
journal = {International Journal of Number Theory},
volume = {14},
number = {03},
pages = {881-898},
year = {2018},
}

@article{murty1988modular, 
title={Modular forms and the Chebotarev density theorem}, author={Murty, M. Ram and Murty, V. Kumar and Saradha, N.}, journal={American Journal of Mathematics}, 
volume={110}, 
number={2}, 
pages={253--281}, 
year={1988}, 
publisher={JSTOR} }

@book{Murty1997,
  author    = {Murty, V. K.},
  title     = {Modular Forms and the Chebotarev Density Theorem II},
  booktitle = {Analytic Number Theory},
  editor    = {Motohashi, Yoichi},
  publisher = {Cambridge University Press},
  year      = {1997},
  pages     = {287--308}
}

@article{Murty2018,
  author = {M. Ram Murty and V. Kumar Murty and Peng-Jie Wong},
  title = {The Chebotarev density theorem and the pair correlation conjecture},
  journal = {Journal of the Ramanujan Mathematical Society},
  volume = {33},
  year = {2018},
  pages = {399-426}
}

@article{amberger2026estimatingnumberzerosdedekind,
  title={Estimating the number of zeros of Dedekind zeta-functions},
  author={Amberger, Victor},
  journal={preprint, arXiv:2510.27444},
  year={2025}
}

@book {Neu99,
    AUTHOR = {Neukirch, J\"urgen},
     TITLE = {Algebraic number theory},
    SERIES = {Grundlehren der mathematischen Wissenschaften [Fundamental
              Principles of Mathematical Sciences]},
    VOLUME = {322},
      NOTE = {Translated from the 1992 German original and with a note by
              Norbert Schappacher,
              With a foreword by G. Harder},
 PUBLISHER = {Springer-Verlag, Berlin},
      YEAR = {1999},
     PAGES = {xviii+571},
      ISBN = {3-540-65399-6},
   MRCLASS = {11Rxx (11-02 11S15 11S31 14C40)},
  MRNUMBER = {1697859},
MRREVIEWER = {Cornelius\ Greither},
       DOI = {10.1007/978-3-662-03983-0},
       URL = {https://doi-org.proxyone.lis.nsysu.edu.tw/10.1007/978-3-662-03983-0},
}

@article{Tur53,
    AUTHOR = {Turing, A. M.},
     TITLE = {Some calculations of the {R}iemann zeta-function},
   JOURNAL = {Proc. London Math. Soc. (3)},
  FJOURNAL = {Proceedings of the London Mathematical Society. Third Series},
    VOLUME = {3},
      YEAR = {1953},
     PAGES = {99--117},
      ISSN = {0024-6115,1460-244X},
   MRCLASS = {65.0X},
  MRNUMBER = {55785},
MRREVIEWER = {D.\ H.\ Lehmer},
       DOI = {10.1112/plms/s3-3.1.99},
       URL = {https://doi-org.proxyone.lis.nsysu.edu.tw/10.1112/plms/s3-3.1.99},
}

@article{Tun81,
    AUTHOR = {Tunnell, Jerrold},
     TITLE = {Artin's conjecture for representations of octahedral type},
   JOURNAL = {Bull. Amer. Math. Soc. (N.S.)},
  FJOURNAL = {American Mathematical Society. Bulletin. New Series},
    VOLUME = {5},
      YEAR = {1981},
    NUMBER = {2},
     PAGES = {173--175},
      ISSN = {0273-0979,1088-9485},
   MRCLASS = {12A67 (10D40 10D45)},
  MRNUMBER = {621884},
MRREVIEWER = {Stephen\ Gelbart},
       DOI = {10.1090/S0273-0979-1981-14936-3},
       URL = {https://doi-org.proxyone.lis.nsysu.edu.tw/10.1090/S0273-0979-1981-14936-3},
}

@book{Lan80,
    AUTHOR = {Langlands, Robert P.},
     TITLE = {Base change for {${\rm GL}(2)$}},
    SERIES = {Annals of Mathematics Studies},
    VOLUME = {No. 96},
 PUBLISHER = {Princeton University Press, Princeton, NJ},
      YEAR = {1980},
     PAGES = {vii+237},
      ISBN = {0-691-08263-4},
   MRCLASS = {10D40 (10-02 12A67 22E55)},
  MRNUMBER = {574808},
MRREVIEWER = {Stephen\ Gelbart},
}

@article{KH09-1,
  title={Serre’s modularity conjecture (I)},
  author={Khare, Chandrashekhar and Wintenberger, Jean-Pierre},
  journal={Inventiones mathematicae},
  volume={178},
  number={3},
  pages={485--504},
  year={2009},
  publisher={Springer}
}

@article{KH09-2,
  title={Serre’s modularity conjecture (II)},
  author={Khare, Chandrashekhar and Wintenberger, Jean-Pierre},
  journal={Inventiones mathematicae},
  volume={178},
  number={3},
  pages={505--586},
  year={2009},
  publisher={Springer}
}

@article{Ra02,
  title={Modularity of solvable Artin representations of GO (4)-type},
  author={Ramakrishnan, Dinakar},
  journal={International Mathematics Research Notices},
  volume={2002},
  number={1},
  pages={1--54},
  year={2002},
  publisher={Hindawi Publishing Corporation}
}

@article{BeWo25,
  title={Improved estimates for the argument and zero-counting function of the Riemann zeta-function},
  author={Bellotti, Chiara and Wong, Peng-Jie},
  journal={Mathematics of Computation},
  year={2025}
}

@article{Gr17,
  title={Explicit zero-counting theorem for Hecke--Landau zeta-functions},
  author={Grze{\'s}kowiak, Maciej},
  journal={Bulletin of the Australian Mathematical Society},
  volume={95},
  number={3},
  pages={400--411},
  year={2017},
  publisher={Cambridge University Press}
}

@article{Pa2019,
  title={On the explicit upper and lower bounds for the number of zeros of the Selberg class},
  author={Paloj{\"a}rvi, Neea},
  journal={Journal of Number Theory},
  volume={194},
  pages={218--250},
  year={2019},
  publisher={Elsevier}
}

\end{document}